\documentclass[12pt,reqno]{amsart}
\usepackage[final]{graphicx}
\usepackage{amsfonts}
\usepackage{pdfsync}
\usepackage{textcomp}
\usepackage{amsmath}
\usepackage{amssymb}
\usepackage{amsthm}

\usepackage{relsize}

\usepackage{amsbsy}
\usepackage{enumerate}
\usepackage[T1]{fontenc}
\usepackage[utf8]{inputenc}
\usepackage{lmodern}

\usepackage{comment}

\newtheorem{theorem}{Theorem}[section]
\newtheorem{lemma}[theorem]{Lemma}

\newtheorem{remark}[theorem]{Remark}

\newtheorem*{remarks}{Remarks}

\newtheorem{claim}[theorem]{Claim}

\def\ue{u_{\epsilon}}
\def\ii{\'\i}
\def\vf{\varphi}
\def\ep{\epsilon}
\def\t{\theta}

\begin{document}

\title[]{\textbf{The Brezis--Nirenberg Problem for the Laplacian with a Singular Drift in $\mathbb{R}^{\lowercase{n}}$ and 
$\mathbb{S}^{\lowercase{n}}$}}

\author[Benguria]{Rafael D. Benguria$^1$}

\author[Benguria]{Soledad Benguria$^2$}

\address{$^1$ Instituto  de F\'\i sica, Pontificia Universidad Cat\' olica de Chile,}
\email{rbenguri@fis.puc.cl}

\address{$^2$ Department of Mathematics, University of Wisconsin - Madison}
\email{benguria@math.wisc.edu}


\smallskip
\begin{abstract}
We consider the Brezis--Nirenberg problem for the Laplacian with a  singular drift for a (geodesic) ball in both  $\mathbb{R}^{n}$ and $\mathbb{S}^n$, $3 \le n \le 5$.
The singular drift we consider 
derives from a potential which is symmetric around the center of the (geodesic) ball. Here the potential is given by a parameter ($\delta$ say) times the logarithm of the 
distance to the center of the ball. In both cases we determine the exact region in the parameter space  for which positive smooth solutions of this 
problem exist and the exact region for which there are no solutions. The parameter space is characterized by the (geodesic) radius of the ball, $\delta$, and $\lambda$, the 
coupling constant of the linear term of the Brezis-Nirenberg problem.
\end{abstract}

\bigskip

\maketitle

\bigskip
\noindent
{\it MCS:} 35XX; 35B33; 35A24; 35J25; 35J60


\section{Introduction}

In 1983, Brezis and Nirenberg \cite{BrNi83} considered the nonlinear eigenvalue problem, 
\begin{equation}
-\Delta u = \lambda u + |u|^{4/(n-2)}u,
\label{standardBN}
\end{equation}
with $u\in H_0^1(\Omega), $ where $\Omega$ is a bounded smooth domain in $\mathbb{R}^n,$ with $n\ge 3$. 
Among other results, they proved that if $n\ge 4,$ there is a positive solution of this problem for all $\lambda\in(0,\lambda_1)$ where $\lambda_1$ is the first Dirichlet eigenvalue of $\Omega$. They also proved that if $n=3,$ there is a $\mu(\Omega)>0,$ such that for any $\lambda\in(\mu,\lambda_1),$ the nonlinear eigenvalue problem has a positive solution, 
whereas if $\lambda \ge \lambda_1$ or $\lambda\le \mu$ there are no positive smooth solutions.  
Moreover, if $\Omega$ is a ball they proved that $\mu = \lambda_1/4$. 

\bigskip

One of the remarkable features of the problem considered by Brezis and Nirenberg is the fact that the boundaries, in the parameter space, that divide the existence and 
nonexistence regions of positive smooth solutions can be sharply determined.

\bigskip
After the publication of  \cite{BrNi83} many authors have considered variants of the problem (\ref{standardBN}). On the one hand, it has been extended to cover domains in
other spaces of constant curvature (see, e.g., \cite{BaBe02} and  \cite{St02} for the analogous problem on domains in $\mathbb{S}^n$ and $\mathbb{H}^n$ respectively). 
On the other hand, different forms of the linear term on the right side of  (\ref{standardBN}) have been explored. Moreover, the Laplacian has been replaced by other linear operators, 
e.g., by the the Laplacian with Hardy perturbation (see, e.g., \cite{CaHa04,CaPe03,ChZo12}). Many other related problems have been considered, in particular the range of values of $\lambda$, for different values of $n$, for which smooth 
sign changing solutions of the radial problem do exist, etc. 

\bigskip
In the Brezis--Nirenberg problem and in all its variants the existence proof relies on a concentration compactness argument while the proof of nonexistence is based 
on a Rellich--Pohozaev argument. Even though these two techniques are in principle unrelated, it is remarkable that they provide a sharp transition in the space of parameters 
of the regions of existence and nonexistence when the domain is a ball in $\mathbb{R}^n$ (or a geodesic ball in $\mathbb{S}^n$ and in $\mathbb{H}^n$, etc.)

\bigskip
What we address in this manuscript is a further variant of the Brezis-Nirenberg problem, namely we study (\ref{standardBN}) with the Laplacian replaced by
a particular form of a weighted Laplacian, or if one prefers, the {\it drift Laplacian}. The interest on weighted Laplacians originated in the early 1980's for different reasons
coming from physics, geometry, and probability. Depending on the context, a weighted Laplacian is often called the {\it Witten Laplacian}  (after \cite{Wi83}) or the 
{\it Bakry-\'Emery Laplacian} (after \cite{BaEm85}). During the past decade there has been a growing interest in studying the spectral properties of weighted Laplacians or  {\it drift Laplacians} (see, e.g., \cite{BeHaNa05, BoJa07, ChLuRo15, HaNaRu11}). As described in \cite{ChLuRo15}, a Bakry--\'Emery manifold, denoted by the triple $(M,g,\phi)$ is a complete Riemannian manifold $(M,g)$ together with some function $\phi \in C^2(M)$  where the measure on $M$ is the weighted measure $\exp(-\phi) \, dV_g$. 
The Bakry--\'Emery Laplacian $\Delta_{\phi}$ associated with such a manifold is given by $\Delta_{\phi}=\Delta-\nabla \phi \cdot \nabla$ which is self-adjoint with respect to the inner product associated with the weighted measure. Here $\Delta$ is the standard Laplace-Beltrami operator and $\nabla$ is the gradient operator on the  Bakry--\'Emery manifold. 
Weighted Laplacians were also introduced, in a different context, by Chavel and Feldman \cite{ChFe91} in the early nineties.

\bigskip
Typically, in the Bakry-\'Emery Laplacian, the potential $\phi$ is smooth, and so is the drift term. However, singular drifts have also been considered in the literature. In fluid mechanics
a weighted Laplacian with a singular drift is rather common, but typically the drift is divergence free, in other words, away from the singularities the potential $\phi$ is harmonic. 
More recently heat kernels with singular drifts have been considered (see, e.g.,  \cite{Gr05,GrOuRo15}). In \cite{GrOuRo15})  a singular drift is considered with a 
(singular) potential of the form $\phi(x)=|x|^{-\alpha}$ with $\alpha>0$. In that case the singular drift has, generically, a nonzero divergence. In our case, when  we consider the 
Brezis-Nirenberg problem for the weighted Laplacian in $\mathbb{R}^n$ ($n \ge 3$), the singular drift derives from a potential of the form $\phi(x) = \delta \log(|x|)$. 
Notice that in this case the weighted measure, $\exp(-\phi) \, dV_g$ becomes $|x|^{-\delta} \, dx$ where $dx$ is the standard Lebesgue measure in $\mathbb{R}^n$. Thus 
the weighted measure can be thought of  as the Lebesgue measure on a space of an effective {\it fractional} dimension $d \equiv  n-\delta$, a fact that we will use intensively 
in the proofs of our theorems. 

\bigskip

\bigskip
In this manuscript we  first consider the problem, 
\begin{equation}
-\Delta u  + \delta \, \frac{\vec x}{|x|^2} \cdot \nabla u = \lambda u + |u|^{4/(n-2-\delta)}u,
\label{BNwithdrift1}
\end{equation}
with $u \in H_0^1(\Omega)$, 
and $\Omega$ is the ball in  $\mathbb{R}^n$, $n\ge 3$, centered at the origin. Equation (\ref{BNwithdrift1}) involves a weighted Laplacian with a singular drift, deriving from the 
potential $\phi(x) = \delta \log(|x|)$, $\delta \in \mathbb{R}$. Because of  Hardy's 
inequality \cite{Da99, Ha919,HLP934}, the operator with the singular drift one considers on the left side of (\ref{BNwithdrift1}) is a positive operator provided $\delta < (n-2)/2$. Notice that the critical Sobolev exponent on the right side of (\ref{BNwithdrift1}) depends on the parameter $\delta$ that characterizes the singular drift. In terms of the ``effective 
dimension'' introduced above, in connection with the definition of the weighted Laplacian we consider here, the critical Sobolev exponent is given by the standard 
form, $(d+2)/(d-2)$, a remark which is important later in the proofs of our theorems. We are interested in the range of values of $\lambda$  and $\delta$ for which 
(\ref{BNwithdrift1}) admits positive radial  smooth solutions.

\bigskip
\noindent
Concerning this problem our main result is the following Theorem.

\begin{theorem}\label{TheTheoremEuclidean}
Let  $\Omega \subset \mathbb{R}^n$, with $3 \le n \le 5$,  be the unit ball centered at the origin. Then, 

\bigskip

\noindent
i) If $n=3$ and  $\delta \in (-1,1/2)$, (\ref{BNwithdrift1}) has a unique positive radial  solution $u \in H_0^1(\Omega)$ provided 
$$
j_{-(1-\delta)/2,1}^2 < \lambda<j_{(1-\delta)/2,1}^2,
$$
and no positive radial solutions for $\lambda$ outside that range. 

\bigskip

\noindent
ii) If $n=4$ and  $\delta \in (0,1)$, (\ref{BNwithdrift1}) has a unique positive radial  solution $u  \in H_0^1(\Omega)$ provided 
$$
j_{-(2-\delta)/2,1}^2 < \lambda<j_{(2-\delta)/2,1}^2,
$$
and no positive radial solutions for $\lambda$ outside that range. 
  
\bigskip

\noindent
iii) If $n=5$ and  $\delta \in (1,3/2)$, (\ref{BNwithdrift1}) has a unique positive radial  solution $u  \in H_0^1(\Omega)$ provided 
$$
j_{-(3-\delta)/2,1}^2 < \lambda<j_{(3-\delta)/2,1}^2,
$$
and no positive radial solutions for $\lambda$ outside that range.

\end{theorem}

\begin{remarks} 
\noindent
a) Here $j_{k,\ell}$ denotes the $\ell$-th positive zero of the Bessel function $J_{k}(t)$.

\noindent
b) Notice that $j_{(n-2-\delta)/2,1}^2$, for $n\ge 3$ and $\delta<(n-2)/2$ is the first Dirichlet eigenvalue of the unit ball, centered at the origin, of the drift Laplacian 
on the left side of (\ref{BNwithdrift1}). 

\noindent
c) Because of the simple behaviour under scaling of the drift Laplacian we consider here, it is trivial to extend the situation to a ball of any radius. In particular the 
first Dirichlet eigenvalue for the ball of radius $R$ becomes $j_{(n-2-\delta)/2,1}^2/R^2$, etc. 
\end{remarks}

\bigskip

\bigskip

Next we consider the Brezis-Nirenberg problem with a singular drift on geodesic balls of $\mathbb{S}^n$ ($n\ge 3$). 
When the underlying manifold is the $n$--dimensional sphere $\mathbb{S}^n$, the standard Brezis--Nirenberg problem is given through  the nonlinear eigenvalue problem 
\begin{equation}
-\Delta_{\mathbb{S}^n}u = \lambda u + |u|^{4/(n-2)}u,
\label{BBSn}
\end{equation}
where $u \in H_0^1(D)$, and $D$ is a geodesic ball in $\mathbb{S}^n$. 
Here $ -\Delta_{\mathbb{S}^n}$ denotes  the Laplace--Beltrami operator in $\mathbb{S}^n$ and
$(n+2)/(n-2)$ is the critical Sobolev exponent. 
In dimension 3, Bandle and Benguria \cite{BaBe02} proved that,  for $\lambda > -3/4$, this problem has a unique positive solution if and only if 
$$ 
\frac{\pi^2 - 4\theta_1^2}{4\theta_1^2}<\lambda< \frac{\pi^2 - \theta_1^2}{\theta_1^2},
$$ 
where $\theta_1$ is the geodesic radius of the ball. Moreover, they proved that if $\lambda \le -3/4$ and $\theta_1 \le \pi/2$ (i.e., for geodesic caps contained in the hemisphere) 
this problem does not have positive radial solutions. It is worth mentioning that for the remaining case, i.e., for $\lambda < -3/4$ and $\pi/2 < \theta_1  \le \pi$,  independently 
Brezis and Peletier \cite{BrPe04,BrPe06} and Bandle and Wei \cite{BaWe07,BaWe08} characterized all the (multiple) positive radial solutions of this problem.

\bigskip

To describe the Brezis--Nirenberg problem for the singular drift Laplacian on $\mathbb{S}^n$ we need some notation. As in \cite{BaBe02} we are only
considering geodesic balls $D$ on $\mathbb{S}^n$ centered at the North Pole ($N\!P$). We denote by $\theta$ the azimuthal angle of a point $P \in D$ (i.e., the angle 
between the vectors that go from the center of $\mathbb{S}^n$ to $P$ and $N\!P$). Here $0 \le \theta \le \theta_1$, where $\theta_1$ 
is the geodesic radius of $D$. It is clear that $0<\theta_1 \le \pi$.

\bigskip
Having this notation in place,  the  Brezis--Nirenberg problem for the singular drift Laplacian on $\mathbb{S}^n$ 
is given by, 
\begin{equation}
-\Delta_{\mathbb{S}^n}u + \delta \,  \frac{\cos \theta}{\sin \theta}  \, \hat \theta \cdot \nabla u = \lambda u + |u|^{4/(n-2-\delta)}u,
\label{BBDriftSn}
\end{equation}
where $u \in H_0^1(D)$, and $D$ is a geodesic ball in $\mathbb{S}^n$, $3 \le n \le 5$.  Equation ({\ref{BBDriftSn}) involves a weighted Laplace--Beltrami operator with a singular drift 
deriving from the potential $\phi(\theta) = \delta \log \sin \theta$, $\delta \in \mathbb{R}$. As in the Euclidean case, because of the appropriate Hardy Inequality (see, e.g., 
\cite{ChRi09}), the operator on the left side  of  ({\ref{BBDriftSn}) is positive--definite provided $\delta<(n-2)/2$. In analogy with the Euclidean case, here the weighted measure, $\exp(-\phi) \, dV_g$ becomes 
${\sin \theta}^{-\delta} \, d\mu$ where $d\mu$ is the invariant measure in $\mathbb{S}^n$. Hence, again in this case, 
the weighted measure can be thought of  as the  measure on a space of an effective {\it fractional} dimension $d \equiv  n-\delta$.

\bigskip

In the sequel, we only consider positive radial solutions (i.e., positive solutions that depend only on the azimuthal angle $\theta$) of (\ref{BBDriftSn}) defined on geodesic caps 
centered at the North--Pole, satisfying Dirichlet boundary conditions, i.e., $u(\theta_1)=0$. In terms of the parameter $d=n-\delta$, the positive radial solutions of (\ref{BBDriftSn}), 
satisfy the ODE, 
\begin{equation}
-u''(\theta) - (d-1) \cot \theta \, u'(\theta) = \lambda u(\theta) + |u(\theta)|^{4/(d-2)}u(\theta),
\label{BBSnRadial}
\end{equation}
where $u$ is such that $u(\theta_1)=0$. Here $' \equiv d/d\theta$, etc.
In what follows we will consider $d$ as just being 
a parameter in equation (\ref{BBSnRadial}), taking values in the interval $(2,4)$.

\bigskip

Our main result concerning (\ref{BBSnRadial}) is the following:

\begin{theorem}\label{TheTheorem} 

\bigskip
\noindent
i) For any $2<d<4$,  if $\lambda \ge -d(d-2)/4$ and $0\le \theta_1 \le \pi$, the boundary value problem (\ref{BBSnRadial}), in the interval $(0,\theta_1)$, with $u'(0)= u(\theta_1)=0$
has a  positive solution if and only if $\lambda$ is such that
\begin{equation}
\mu \equiv \frac{1}{4} [(2 \ell_2 +1)^2 - (d-1)^2] < \lambda < \frac{1}{4} [(2 \ell_1 +1)^2 - (d-1)^2] \equiv \lambda_1,
\label{rangeoflambda}
\end{equation}
where $\ell_1$ (respectively $\ell_2$) is the first positive value of $\ell$ for which the associated Legendre function ${\rm P}_{\ell}^{(2-d)/2} (\cos\theta_1)$ (respectively
 ${\rm P}_{\ell}^{(d-2)/2} (\cos\theta_1)$) vanishes.

\bigskip
\noindent
ii) Moreover, for any $d>2$, if $\lambda \le  -d(d-2)/4$ and $0\le \theta_1 \le \pi/2$, the boundary value problem (\ref{BBSnRadial}), in the interval $(0,\theta_1)$, with $u'(0)= u(\theta_1)=0$
does not have a positive solution. 
\end{theorem}
\begin{remark} 
a) In the remaining sector, i.e., for $\lambda < -d(d-2)/4$ and $\pi/2 < \theta_1  \le \pi$,  for any $2<d<4$ one expects to have multiple solutions to this problem in a similar vein as in the case $d=3$ studied in \cite{BaWe07,BaWe08,BrPe04,BrPe06}. 
b) In i)  the positive solution, if it  exists, is unique. 
\end{remark}

\bigskip
\noindent
We illustrate the results of Theorem \ref{TheTheorem} in the following two figures. In figure 1, we show, for $\theta_1=\pi/3$  the region (shaded in the figure) of existence 
of positive solutions for the parameter $\lambda$ as a function of $d$ (in this figure we only illustrate what happens for $\lambda > -d(d-2)/4$). Notice how the width of the shaded region increases with $d$ and that there is no gap (between $\lambda$ and $-d(d-2)/4$) when $d=4$. 

\begin{figure}[h!]
    \centering
    \includegraphics[width=0.4 \textwidth]{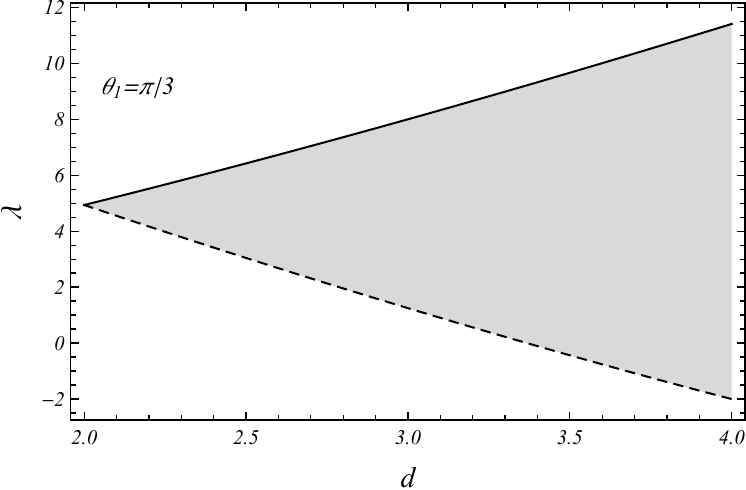}
    \label{monitopitercio}
\caption{The region in the $\lambda$--$d$ space, for $\theta_1=\pi/3$, where radial positive solutions of (\ref{BBSnRadial}) exist.}
\end{figure}

Finally, the region in Figure 2 between the curves $\mu$ and $\lambda_1$, is the region in the $\lambda$--$\theta_1$ space, for fixed $d=3.5$, where radial positive solutions of (\ref{BBSnRadial}) exist.

\begin{figure}[h!]
    \centering
    \includegraphics[width=0.5\textwidth]{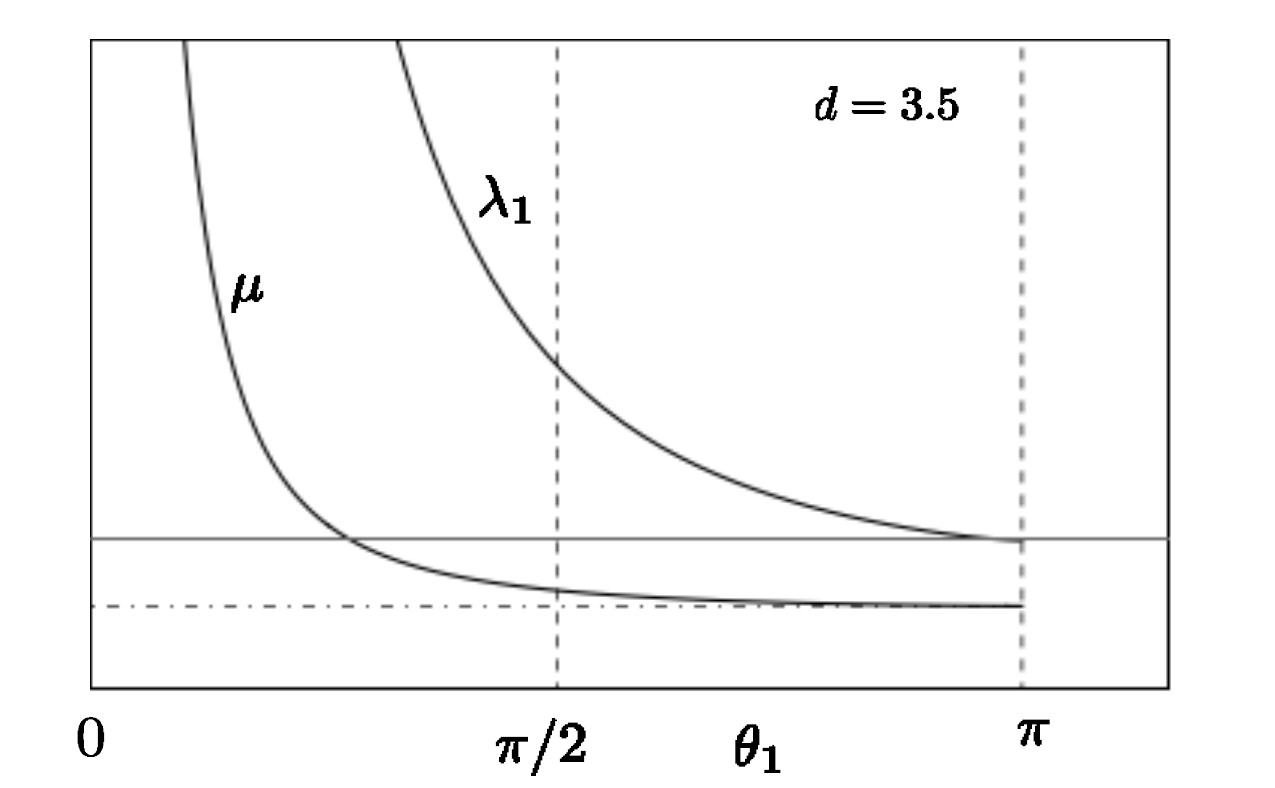}
    \label{monitoconNfijo}
\caption{The region in the $\lambda$--$\theta_1$ space, for fixed $d=3.5$, where radial positive solutions of (\ref{BBSnRadial}) exist.}
\end{figure}

\bigskip
For the drift Laplacian on domains of $\mathbb{S}^n$, $n \ge 3$, the results contained in Theorem 1.2 can be cast in the following form
(for the proof, see Appendix A). 

\begin{theorem}\label{TheTheoremDriftSn}
Let  $D \subset \mathbb{S}^n$, with $3 \le n \le 5$,  be a geodesic ball of geodesic radius $0<\theta_1 \le \pi$ centered at the 
North--Pole, and let  $\lambda \ge -(n-\delta)(n-2-\delta)/4$. Then we have:

\bigskip

\noindent
i) If $n=3$ and  $\delta \in (-1,1/2)$, (\ref{BBDriftSn}) has a unique positive radial  solution $u \in H_0^1(D)$ provided 
$$
\frac{1}{4} [(2 \ell_2 +1)^2 - (2-\delta)^2] < \lambda < \frac{1}{4} [(2 \ell_1 +1)^2 - (2-\delta)^2], 
$$
where $\ell_1$ (respectively $\ell_2$) is the first positive value of $\ell$ for which the associated Legendre function 
${\rm P}_{\ell}^{(\delta-1)/2} (\cos\theta_1)$ (respectively
 ${\rm P}_{\ell}^{(1-\delta)/2} (\cos\theta_1)$) vanishes.
Moreover, (\ref{BBDriftSn}) has no positive radial solutions for $\lambda$ outside that range.

\bigskip

\noindent
ii) If $n=4$ and  $\delta \in (0,1)$, (\ref{BBDriftSn}) has a unique positive radial  solution $u \in H_0^1(D)$ provided 
$$
\frac{1}{4} [(2 \ell_2 +1)^2 - (3-\delta)^2] < \lambda < \frac{1}{4} [(2 \ell_1 +1)^2 - (3-\delta)^2], 
$$
where $\ell_1$ (respectively $\ell_2$) is the first positive value of $\ell$ for which the associated Legendre function 
${\rm P}_{\ell}^{(\delta-2)/2} (\cos\theta_1)$ (respectively
 ${\rm P}_{\ell}^{(2-\delta)/2} (\cos\theta_1)$) vanishes.
Moreover, (\ref{BBDriftSn}) has no positive radial solutions for $\lambda$ outside that range.

\bigskip

\noindent
iii) If $n=5$ and  $\delta \in (1,3/2)$, (\ref{BBDriftSn}) has a unique positive radial  solution $u \in H_0^1(D)$ provided 
$$
\frac{1}{4} [(2 \ell_2 +1)^2 - (4-\delta)^2] < \lambda < \frac{1}{4} [(2 \ell_1 +1)^2 - (4-\delta)^2], 
$$
where $\ell_1$ (respectively $\ell_2$) is the first positive value of $\ell$ for which the associated Legendre function 
${\rm P}_{\ell}^{(\delta-3)/2} (\cos\theta_1)$ (respectively
 ${\rm P}_{\ell}^{(3-\delta)/2} (\cos\theta_1)$) vanishes.
Moreover, (\ref{BBDriftSn}) has no positive radial solutions for $\lambda$ outside that range.

\end{theorem}
\begin{remark} 

\noindent
i) For any $n\ge 3$, and $\delta < (n-2)/2$, if $\lambda < -(n-\delta) (n-2-\delta)/4$ and $0 < \theta_1 \le \pi/2$, (\ref{BBDriftSn}) has no positive radial solutions. 

\noindent
ii) In the remaining sector, i.e., for $\lambda < -(n-\delta) (n-2-\delta)/4$ and $\pi/2 < \theta_1  \le \pi$,  for $n=3, 4, 5$, 
one expects to have multiple solutions to this problem in a similar vein as in the case $n=3$, $\delta=0$ studied in \cite{BaWe07,BaWe08,BrPe04,BrPe06}. 
\end{remark}

\bigskip
\bigskip
The rest of the manuscript is organized as follows. Sections 2, 3 and 4 are dedicated to prove Theorem 1.2. 
In section \ref{sec:prelimianries} we begin by showing that $\ell_2<\ell_1$. That is, the range of existence of radial positive solutions of 
(\ref{BBSnRadial}) given by (\ref{rangeoflambda}) is non empty. We then show that the upper bound corresponds to the first Dirichlet eigenvalue of the geodesic ball. That is, we show that if $\lambda_1$ is the first positive eigenvalue of the boundary value problem 
$$
-u''(\theta) - (d-1) \cot \theta \, u'(\theta)  = \lambda u(\theta)
$$ 
with $u(\theta_1)=0$,  then $\lambda_1 =  \frac{1}{4} [(2 \ell_1 +1)^2 - (d-1)^2]$. 

\bigskip
In section \ref{sec:existencia} we show that there are positive radial solutions if $\frac{1}{4} [(2 \ell_2 +1)^2 - (d-1)^2] < \lambda < \frac{1}{4} [(2 \ell_1 +1)^2 - (d-1)^2]$, and in section \ref{sec:noexistencia}  we show that there are no positive radial solutions if $\lambda \le \frac{1}{4} [(2 \ell_2 +1)^2 - (d-1)^2].$ 

\bigskip 
The proof of Theorem 1.1 follows easily from the result of Jannelli \cite{Ja99} while the proof of Theorem 1.4 follows from our Theorem 1.2. We give the proofs of Theorem 1.1 and 1.4 
in the Appendix A. 

\bigskip
The analog of Theorem 1.2 in the hyperbolic space $\mathbb{H}^d$, $2<d<4$ has been recently proved by one of us in \cite{Be16}. The results in \cite{Be16} can 
also be  expressed in terms of an appropriate singular drift Laplace--Beltrami operator in $\mathbb{H}^n$, for $n=3,4,5$ and for special values of the coupling constant of the drift. 

 \bigskip
 \bigskip

\section{Preliminaries}\label{sec:prelimianries}

We begin by studying the order of the first positive zeroes of $P_\ell^{\nu}(s)$ and $P_\ell^{-\nu}(s)$ respectively, where $\nu \in (0,1).$

\begin{lemma}\label{lemma2.1}
Let $\alpha= (2-d)/2, $ with $2<d<4.$ Let $\t_1\in\left(0,\frac{\pi}{2}\right)$ be fixed and choose $\ell_1$ (respectively $\ell_2$) to be the first positive value of $\ell$ for which the associated Legendre function ${\rm P}_{\ell}^{(2-d)/2} (\cos\theta_1)$ (respectively
 ${\rm P}_{\ell}^{(d-2)/2} (\cos\theta_1)$) vanishes. Then $\ell_2 < \ell_1.$ 
\end{lemma}

\begin{proof}
Let $y_1 = P_{\ell_1}^{\alpha}(\cos\t)$ and $y_2=P_{\ell_2}^{-\alpha}(\cos\t).$ Then $y_1$ and $y_2$ satisfy the equations
\begin{equation}\label{eq:dey1}
y_1'' +\cot\t y_1' + \left( \ell_1(\ell_1+1) - \frac{\alpha^2}{\sin^2\t}\right)y_1 = 0,
\end{equation}
and 
\begin{equation}\label{eq:dey2}
y_2'' +\cot\t y_2' + \left( \ell_2(\ell_2+1) - \frac{\alpha^2}{\sin^2\t}\right)y_2 = 0
\end{equation}
respectively.

\bigskip

Let $W  = y_1' \, y_2 - y_2'\, y_1$ be  the Wronskian of $y_2$ and $y_1$.   Then $ W' = y_1''\,  y_2 - y_2'' \,  y_1$.
 Multiplying equation (\ref{eq:dey1}) by $y_2$ and equation (\ref{eq:dey2}) by $y_1$ and substracting, it follows that 
\begin{equation}
(\sin\t  \, W)'  + (\Delta_1- \Delta_2)\sin\t  \, y_1 \, y_2 = 0,
\label{wronskian}
\end{equation}
where $\Delta_1 =  \ell_1(\ell_1+1)$ and $\Delta_2= \ell_2(\ell_2+1)$.  To prove the lemma it suffices to show that $\Delta_1 > \Delta_2$.  

\bigskip
Integrating (\ref{wronskian}) in $\theta$  between $0$ and $\t_1$,  we get, 
\begin{equation}
\sin\t_1 W(\t_1) - \lim_{\t\rightarrow 0} \sin\t  \, W(\t) + (\Delta_1-\Delta_2) C = 0
\label{intwronskian}
\end{equation}
where $C = \displaystyle\int_0^{\t_1} \sin\t   \, y_1(\t) \,  y_2(\t) \, d\t >0$ by hypothesis. Since $W(\t_1) = 0,$ it suffices to show that $\lim_{\t\rightarrow 0} \sin\t W(\t) >0.$ 
The series expansion of the associated Legendre functions around $\theta=0$  is given  by 
\begin{equation}
 P_{\ell}^{\nu}(\cos\t)  =\frac{1}{\Gamma(1-\nu)}\left(\cot\frac{\t}{2}\right)^\nu \,_2F_1\left(-\ell,\ell+1,1-\nu,\sin^2\frac{\t}{2}\right), 
 \label{Taylor}
 \end{equation}
in terms of the hypergeometric function, 
\begin{equation}
_2F_1(\sigma,\beta,\gamma,z) = \frac{\Gamma(\gamma)}{\Gamma(\sigma)\Gamma(\beta)}\sum_{k=0}^{\infty}\frac{\Gamma(k+\sigma)\Gamma(k+\beta)}{\Gamma(k+\gamma)n!}z^k.
\label{hyper}
\end{equation}
From (\ref{Taylor}) and (\ref{hyper}), and using that  $-1<\alpha<0$,  the behavior of $y_1$, $y_2,$ $y_1'$ and $y_2'$ in a neighborhood of the origin to leading order is given by
$$ 
y_1 \approx \frac{1}{\Gamma(1-\alpha)}\left(\cot\frac{\t}{2}\right)^{\alpha}, 
$$
$$ 
y_2 \approx \frac{1}{\Gamma(1+\alpha)}\left(\cot\frac{\t}{2}\right)^{-\alpha}, 
$$
$$ 
y_1' \approx  \frac{\alpha}{\Gamma(1-\alpha)}  \left(\cot\frac{\t}{2}\right)^{\alpha-1}\left( \frac{-1}{2\sin^2\frac{\t}{2}}\right),
$$
and 
$$ 
y_2' \approx - \frac{\alpha}{\Gamma(1+\alpha)} \left(\cot\frac{\t}{2}\right)^{-\alpha-1}\left( \frac{-1}{2\sin^2\frac{\t}{2}}\right).
$$
Using this behavior of $y_1(\theta)$,  $y_2(\theta)$,  $y_1'(\theta)$, and  $y_2'(\theta)$,  for small $\theta$, after some calculations we get 
\begin{equation}
\lim_{\t \to  0}\sin\t \, W(\t) = \frac{2}{\pi} \sin \left(\frac{\pi(d-2))}{2} \right) > 0, 
\label{boundarywronskian}
\end{equation}
for all  $2<d<4$. To obtain (\ref{boundarywronskian}) we have used that $\alpha =(2-d)/2$ and the fact that
$$
\Gamma(1+\alpha) \, \Gamma(1-\alpha) = \frac{\pi \alpha}{\sin (\pi \, \alpha)}.
$$
\end{proof}

\begin{lemma}
 Let $\lambda_1$ be the first positive eigenvalue of 
\begin{equation}
-u''(\theta) - (d-1) \cot \theta \, u' (\theta) = \lambda u(\theta)
\label{eq:a1}
\end{equation} 
in the interval $(0,\theta_1)$ with $u'(0)=0$ and $u(\theta_1)=0$. 
Then, 
$$
\lambda_1 =  \frac{1}{4} [(2 \ell_1 +1)^2 - (d-1)^2],
$$ 
where $\ell_1$ is the first positive value of $\ell$ for which the associated Legendre function ${\rm P}_{\ell}^{(2-d)/2} (\cos\theta_1)$ vanishes.  
\end{lemma}
\begin{proof}
Let $\alpha=(2-d)/2$, and set 
\begin{equation}
u(\theta) = (\sin \theta)^{\alpha} \, v(\theta).
\label{eq:a2}
\end{equation}
Then $v(\theta)$ satisfies the equation, 
\begin{equation}
v''(\theta) + \frac{\cos \theta}{\sin \theta} v'(\theta) +   \left(\lambda_1 +\alpha(\alpha-1) - \frac{\alpha^2}{\sin^2 \theta} \right) v (\theta)= 0. 
\label{eq:a3}
\end{equation}
In the particular case when $d=3$, $\alpha=-1/2$ and this equation becomes, 
\begin{equation}
v''(\theta) + \frac{\cos \theta}{\sin \theta} v'(\theta) +  \left(\lambda_1 + \frac{3}{4} - \frac{1}{4 \, \sin^2 \theta} \right) v (\theta) = 0,
\label{eq:a4}
\end{equation}
whose positive regular solution is given by, 
\begin{equation}
v(\theta) = C \, \frac{\sin \left(\sqrt{1+\lambda_1}  \, \theta \right)}{\sqrt{\sin \theta}}.
\label{eq:a5}
\end{equation}
Hence, in this case, 
\begin{equation}
u(\theta) = C \, \frac{\sin \left(\sqrt{1+\lambda_1}  \, \theta \right)}{\sin \theta}.
\label{eq:a6}
\end{equation}
Imposing the boundary condition $u (\theta_1)=0$, in the case $d=3$, we find that, 
\begin{equation}
\lambda_1 (\theta_1) = \frac{\pi^2- \theta_1^2}{{\theta_1}^2}. 
\label{eq:a7}
\end{equation}
Now, for any $2<d<4$ the solutions of (\ref{eq:a4}) are $P_{\ell}^\alpha(\cos\t)$ and $P_{\ell}^{-\alpha}(\cos\t),$
with 
\begin{equation}
 \alpha = (2-d)/2, 
\label{eq:a9}
\end{equation}
and $\ell$ the positive root of 
\begin{equation}
\ell(\ell+1)  = \lambda_1 + \alpha(\alpha-1),
\label{eq:a10}
\end{equation}
that is, 
$$
\ell = \frac{1}{2}\left( \sqrt{4\lambda_1 +(d-1)^2}-1\right).
$$ 
Taking into account (\ref{Taylor}) and (\ref{hyper}) we see that 
the regular solution of (\ref{eq:a1}) is given by
\begin{equation}
u(\t) = \sin^{\alpha}\t \,  P_{\ell}^\alpha(\cos\t).
\label{DirEF}
\end{equation}
Finally, the boundary conditions $u(\t_1) = 0$ and $u(\t)>0$ if $0\le \t <\t_1$ imply that $\ell= \ell_1$, and so 
$$
\lambda_1 =  \frac{1}{4} [(2 \ell_1 +1)^2 - (d-1)^2].
$$
Here, $\ell_1$ is the first positive value of $\ell$ for which the associated Legendre function ${\rm P}_{\ell}^{(2-d)/2} (\cos\theta_1)$ vanishes.  
\end{proof}

\bigskip
\bigskip

\section{Existence of solutions}\label{sec:existencia}

Let $D$ be a geodesic ball on $\mathbb{S}^n.$ If $n$ is a natural number, the solutions of 

\begin{equation}\label{eq:111} \left\{ \begin{array}{lcr}
\mbox - \Delta_{\mathbb{S}^n} u = \lambda u + u^{p}&\text{on}&D\\
\mbox u>0 & \text{on} &D  \\
\mbox u=0 & \text{on} & \partial D, \end{array}\right.
\end{equation}

\noindent where $p=\frac{n+2}{n-2}$ correspond to minimizers of 

\begin{equation}\label{eq:Q}
Q_{\lambda}(u) = \frac{\displaystyle\int_{D}(\nabla u)^2q^{n-2}\,dx - \lambda \displaystyle\int_{D} u^2 q^n \, dx}{\left(\displaystyle\int_{D} u^\frac{2n}{n-2}q^n\,dx\right)^\frac{n-2}{n}}. 
\end{equation}

\noindent Here $q(x) = \frac{2}{1+|x|^2}, $ so that $ds = q(x) dx$ is the line element of $\mathbb{S}^n;$  and $x\in D',$ where $D'$ is the projection of the stereographic ball (see Figure 3).

\begin{figure}[h!]
    \centering
    \includegraphics[width=0.4\textwidth]{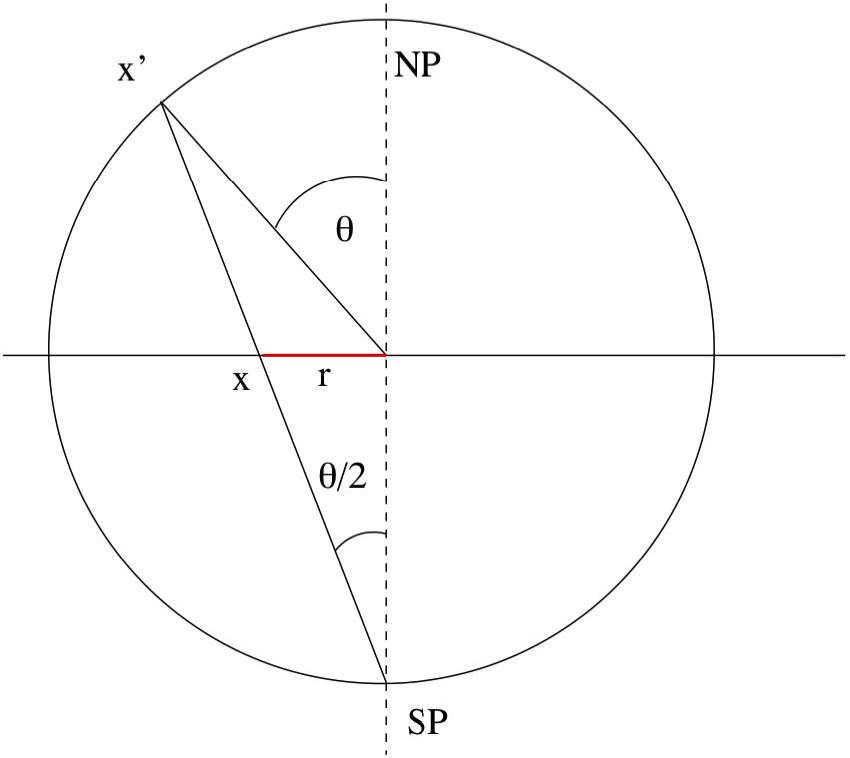}
    \label{estereografica}
    \caption{Stereographic mapping of a geodesic cap into a domain in the equatorial plane.}
\end{figure}

\bigskip 

 If $u$ is radial, then even for fractional $n$, which we will denote henceforth by $d$ to make it consistent with our notation in the Introduction,  we can write 
 
 \begin{equation}\label{eq:Qr}
Q_{\lambda}(u) = \frac{ \omega_d\displaystyle\int_0^R r^{d-1}q(r)^{d-2} u'(r)^2\, dr - \lambda  \omega_d \displaystyle\int_0^R   r^{d-1} q(r)^d u^2(r)\, dr}{\left(\omega_d \displaystyle\int_0^R r^{d-1} q(r)^d u(r)^\frac{2d}{d-2}\, dr \right)^\frac{d-2}{d}}. 
\end{equation}

 \noindent Here $R$ corresponds to the stereographic projection of $\theta_1.$

As in \cite{BaBe02}, let 

\begin{equation}\label{eq:Sl}
S_{p,\lambda}(D) = \inf_{\substack{ u\in H_0^1 \\  ||u||_{p+1} = 1}} \{ ||\nabla u ||_2^2 - \lambda ||u||_2^2 \},
\end{equation}

\noindent so that $S_{p,\lambda}\le Q_{\lambda}(u),$ and let 

\begin{equation}\label{eq:S}
S  = \inf_{\substack{ u\in H_0^1 \\  ||u||_{p+1} = 1}}{||\nabla u||_2^2}. 
\end{equation}
Here, we are abusing notation and  still  using $||\nabla u ||_2^2$, etc.,  to mean  $\omega_d \int_0^R r^{d-1} \, q(r)^{d-2} u'(r)^2\, dr$, etc.,  when $d$ is fractional. 
By the Brezis--Lieb compactness lemma \cite{BrLi83}, it is known that in $\mathbb{R}^d,$ if there is a function that satisfies $Q_{\lambda}(u)<S, $ then the minimizer for $Q_{\lambda}$ is attained. The minimizer is positive and satisfies the Brezis--Nirenberg equation.

\begin{lemma} 

Let 2<d<4 and 
$$
\frac{1}{4} [(2 \ell_2 +1)^2 - (d-1)^2] < \lambda < \frac{1}{4} [(2 \ell_1 +1)^2 - (d-1)^2] ,
$$
where $\ell_1$ (respectively $\ell_2$) is the first positive value of $\ell$ for which the associated Legendre function ${\rm P}_{\ell}^{(2-d)/2} (\cos\theta_1)$ (respectively
 ${\rm P}_{\ell}^{(d-2)/2} (\cos\theta_1)$) vanishes. Then there is a unique  positive solution to 
 \begin{equation}
-u''(\theta) - (d-1) \cot \theta \, u'(\theta)  = \lambda u(\theta) + u(\theta)^{(d+2)/(d-2)}
\label{eq:ODE1}
\end{equation} 
with $u'(0)=u(\theta_1)=0$. 
\end{lemma}

\begin{remark} 
In this Section we only give the existence proof. The proof of uniqueness follows from a general result of Kwong and Li \cite{KwLi92} and it is sketched in Appendix B. 
\end{remark}

\begin{proof}
It suffices to show that there exists $u\in H_0^1(D)$ such that $ Q_{\lambda}(u)< S$. 
Let $\varphi$ be a smooth function such that $\varphi (0)=1,$ $\varphi '(0) = 0$ and $\varphi (R) = 0, $ where $R$ is the stereographic projection of $\theta_1.$  For $\epsilon>0,$ let
\begin{equation}\label{eq:u}
u_{\epsilon}(r) = \frac{\varphi (r)}{(\epsilon+r^2)^\frac{d-2}{2}}.
 \end{equation}

\noindent We claim that for $\epsilon$ small enough, $Q_{\lambda}(u_{\epsilon}) < S$. 
In the next three claims we compute $||\nabla u_{\epsilon}||_2^2,$ $||u_{\epsilon}||_{p+1}^2$ and $||u_{\epsilon}||_2^2.$

\bigskip

\begin{claim}\label{cl:primera}

\begin{equation}\label{eq:cl11}\begin{split}
 \omega_d\displaystyle\int_0^R r^{d-1}q(r)^{d-2} \ue'(r)^2\, dr = &\, \omega_d \displaystyle\int_0^R \varphi '(r)^2r^{3-d}q(r)^{d-2}\,dr -\omega_d (d-2)^2 \displaystyle\int_0^R \varphi(r)^2r^{3-d}q(r)^{d-1}\, dr\\
&+ \omega_d d(d-2)2^{d-2}D_d\epsilon^{\frac{2-d}{2}}+  \mathcal{O}(\epsilon^\frac{4-d}{2}),\\
\end{split}
\end{equation}
where
\begin{equation}
 D_d = \frac{1}{2}\frac{\Gamma\left(\frac{d}{2}\right)^2}{\Gamma(d)},  \qquad \omega_d = \frac{2\pi^\frac{d}{2}}{\Gamma\left(\frac{d}{2}\right)}.
\end{equation}
\end{claim}

\begin{proof}
Let 
$$
I(\epsilon) =  \omega_d\displaystyle\int_0^R r^{d-1}q(r)^{d-2} \ue'(r)^2\, dr.
$$
Then 
\begin{equation}\label{eq:cl12}\begin{split}
I(\epsilon) =& \,  \omega_d\displaystyle\int_0^R r^{d-1}q^{d-2}\left(\frac{\varphi'^2}{(\epsilon+r^2)^{d-2}} -\frac{2(d-2)r\varphi\varphi'}{(\epsilon+r^2)^{d-1}} +\frac{r^2\vf^2(d-2)^2}{(\ep+r^2)^d}  \right)\, dr.\\
\end{split}\end{equation}
Integrating by parts the term with $\vf\vf',$ we obtain $I(\epsilon) = I_1+ I_2+I_3,$ where 
$$ 
I_1(\ep) = \omega_d \displaystyle\int_0^R r^{d-1}q^{d-2}  \frac{\vf'^2}{(\ep+r^2)^{d-2}}\,dr;
$$
$$ 
I_2(\ep) = \omega_d(d-2)^2 \displaystyle\int_0^R r^d q^{d-3} q'\,\frac{\vf^2}{(\epsilon+r^2)^{d-1}}\, dr;
$$
and 
$$ 
I_3(\ep) = \omega_d (d-2)d \ep \displaystyle\int_0^R q^{d-2}r^{d-1} \frac{\vf^2}{(\epsilon+r^2)^d}.
$$
We begin by showing that 
$$
I_1(\ep) =  \omega_d \displaystyle\int_0^R r^{3-d}q^{d-2}\vf'^2\, dr  +\mathcal{O}(\epsilon).
$$
Notice that 
$$
I_1(0) = \omega_d \displaystyle\int_0^R r^{3-d}q^{d-2}\vf'^2\, dr 
$$
\noindent converges for $d<4.$ It suffices to show that $I_1(\epsilon) - I_1(0) = \mathcal{O}(\epsilon).$ We can write 
$$ 
I_1(\epsilon) - I_1(0) = \omega_d \displaystyle\int_0^R r^{d-1} q^{d-2} \vf'^2 \displaystyle\int_0^\ep \frac{d-2}{(a+r^2)^{d-1}}\, da \, dr.
$$ 
But 
$$
\displaystyle\int_0^R q^{d-2} r^{d-1} \frac{\vf'^2}{(a+r^2)^{d-1}}\, dr \le  C \displaystyle\int_0^R 2^{d-2} r^{3-d}\, dr,
$$
for some constant $C$, which converges if $d<4,$ thus yielding the desired result. 

\bigskip

\noindent
Next let us consider $I_2.$ We will show that 
$$ 
I_2(\ep) = -\omega_d (d-2)^2 \displaystyle\int_0^R q^{d-1}\vf^2 r^{3-d}\, dr + \mathcal{O}(\epsilon^{\frac{4-d}{2}}).
$$ 
Notice that $q' = -q^2 r$,  so that 
$$ 
I_2(\epsilon)  = -\omega_d (d-2)^2 \displaystyle\int_0^R q^{d-1}r^{d+1}\frac{\vf^2}{(\ep+r^2)^{d-1}}\, dr.
$$ 
As in the previous integral, let $ I_2(\epsilon) = I_2(0) + I_2(\ep)-I_2(0)$. 
Then it suffices to show that $I_2(\ep)-I_2(0) =  \mathcal{O}(\epsilon^{\frac{4-d}{2}})$.
We can write 
\begin{equation}
\begin{split} 
I_2(\ep)-I_2(0) =&\,  \omega_d (d-2)^2 \displaystyle\int_0^R \vf^2 q^{d-1}r^{d+1} \left( \frac{1}{r^{2d-2}} - \frac{1}{(\ep+r^2)^{d-1}}\right) \, dr \\
=&\, \omega_d (d-2)^2 \displaystyle\int_0^R q^{d-1}r^{d+1}\left[(\vf^2-1) + 1\right] \displaystyle\int_0^{\ep} \frac{d-1}{(a+r^2)^d}\, da\, dr.\\
\end{split}
\end{equation}
Let 
$$ 
I_{21}(\epsilon) = \displaystyle\int_0^R q^{d-1}r^{d+1}\displaystyle\int_0^\epsilon \frac{d-1}{(a+r^2)^d}\, da\, dr,
$$
and 
$$ 
I_{22}(\ep) = \displaystyle\int_0^R q^{d-1} r^{d+1} (\vf^2-1) \displaystyle\int_0^{\ep} \frac{d-1}{(a+r^2)^d}\, da\, dr.
$$
Then, since $q^d \le 2^d$,  and making the change of variables $r=s \, \sqrt{a},$ it follows that 
$$ 
I_{21}(\ep) \le 2^{d-1}(d-1) \displaystyle\int_0^\ep a^\frac{2-d}{2} \displaystyle\int_0^{\infty} \frac{s^{d+1}}{(1+s^2)^d}\, ds\, da.
$$ 
The inner integral converges if $d>2,$ so it follows that 
$$ I_{21}(\ep)  = \mathcal{O}(\epsilon^\frac{4-d}{2}). 
$$
Also, since by hypothesis $\vf(0)=1$ and $\vf'(0) = 0,$ it follows that $\vf^2-1 \le Cr^2.$ Thus, 
$$ 
I_{22}(\ep)  \le C \,  2^{d-1}(d-1) \displaystyle\int_0^\ep  \displaystyle\int_0^R   r^{3-d} \, dr \, da.
$$
The inner integral converges if $d<4,$ so it follows that $I_{22}(\ep) = \mathcal{O}(\ep)$. 
In particular, since $d \ge 2$,  $I_{22}(\ep) = \mathcal{O}(\epsilon^\frac{4-d}{2})$ and 
$$ 
I_2(\ep)-I_2(0)   = \mathcal{O}(\epsilon^\frac{4-d}{2}). 
$$

\bigskip 

\noindent
Finally, we must show that 
$$
I_3(\ep) =  \omega_d  \, d(d-2)2^{d-2} \, D_d \epsilon^{\frac{2-d}{2}}+  \mathcal{O}(\epsilon^\frac{4-d}{2}).
$$ 
Writing 
$$ 
q^{d-2}\vf^2 = q^{d-2}(\vf^2-1) + (q^{d-2}-2^{d-2}) + 2^{d-2},
$$
we have that $I_3 = \omega_d  \, (d-2) \, d \, (I_{31}+I_{32}+ I_{33})$,  where 
$$
I_{31}  =  \displaystyle\int_0^R \frac{\ep \, r^{d-1}q^{d-2}(\vf^2-1)}{(\ep+r^2)^d}\, dr; 
$$
$$ 
I_{32} = \displaystyle\int_0^R \frac{\ep\, r^{d-1} (q^{d-2}- 2^{d-2}) }{(\ep+r^2)^d} \, dr; 
$$
and 
$$ 
I_{33} =  2^{d-2} \displaystyle\int_0^R \frac{\ep\, r^{d-1} }{(\ep+r^2)^d}\, dr.
$$
As before, since $\vf^2-1 \le C \, r^2$,  it follows that 
$$ 
I_{31} \le C \, 2^{d-2} \ep \displaystyle\int_0^R \frac{r^{d+1}}{(\ep+r^2)^d}\, dr.
$$
Letting $ r= s \, \sqrt{\ep}$,  it follows that 
\begin{equation}\label{eq:rn}
\displaystyle\int_0^R \frac{r^{d+1}}{(\ep+r^2)^d}\, dr \le \ep^{\frac{2-d}{2}} 
\displaystyle\int_0^{\infty} \frac{s^{d+1}}{(1+s^2)^d}\, ds = \mathcal{O}(\ep^\frac{2-d}{2}),
\end{equation}
since the integral converges for all $d>2$. Thus, 
$$ 
I_{31} =  \mathcal{O}(\ep^\frac{4-d}{2}).
$$ 
Similarly, and since if $0\le r\le R$ then $2^{d-2} - q^{d-2} \le 2^{d-2} A(R) r^2$,  with $A(R) = (d-2)(1+R^2)^{d-3}$, we have that 
$$ 
|I_{32}| \le 2^{d-2} \, A(R) \, \epsilon \displaystyle\int_0^R \frac{r^{d+1}}{(\ep+r^2)^d}\, dr =  \mathcal{O}(\ep^\frac{4-d}{2}) .
$$
Finally, making the change of variables $r = s \, \sqrt{\ep}$,  it follows that 
$$
I_{33}  = 2^{d-2}\ep^\frac{2-d}{2} \left[ \displaystyle\int_0^{\infty} \frac{s^{d-1}}{(1+s^2)^d} \, ds -  
\displaystyle\int_\frac{R}{\sqrt{\ep}}^{\infty} \frac{s^{d-1}}{(1+s^2)^d}\, ds\right].
$$
But 
\begin{equation}\label{eq:errorcola}
\displaystyle\int_\frac{R}{\sqrt{\ep}}^{\infty} \frac{s^{d-1}}{(1+s^2)^d} \, ds \le  \displaystyle\int_\frac{R}{\sqrt{\ep}}^{\infty} s^{-d-1} \, ds = \mathcal{O}(\epsilon^\frac{d}{2}).
\end{equation}
Moreover, notice that making the change of variables $u = s^2,$ we can write
 \begin{equation}\label{eq:casigama}
  \displaystyle\int_0^{\infty} \frac{s^{d-1}}{(1+s^2)^d}\, ds = \frac{1}{2} 
  \displaystyle\int_0^{\infty} \frac{u^{\frac{d}{2}-1}}{(1+u)^d}\,du = \frac{1}{2} \frac{\Gamma\left(\frac{d}{2}\right)^2}{\Gamma(d)} = D_d.
 \end{equation}
 
\bigskip
\noindent
Here we have used the standard integral 
\begin{equation}\label{eq:integralgamma}
\displaystyle\int_0^{\infty} \frac{x^{k-1}}{(1+x)^{k+m}}\, dx  = \frac{\Gamma(k)\Gamma(m)}{\Gamma(k+m)}
\end{equation}
(see, e.g., \cite{Dw61}, equation 856.11, page 213), which holds for all $m>0$, and for all  $k\,>0$.  Thus, 
$$
I_{33}  = 2^{d-2}\ep^\frac{2-d}{2} D_d + \mathcal{O}(\ep).
$$
This yields the desired estimate for $I_3$. 
\end{proof}

\bigskip

\begin{claim}\label{cl:segunda}
$$
\omega_d  \displaystyle\int_0^R   r^{d-1} \, q^d \, u^2\, dr= \omega_d \displaystyle\int_0^R q^d \,  r^{3-d} \vf^2 \, dr + \mathcal{O}(\ep^\frac{4-d}{2}).
$$ 
\end{claim}

\begin{proof}
Let 
$$
J(\ep) =  \omega_d \displaystyle\int_0^R   r^{d-1} q^d \frac{\vf^2}{(\ep+r^2)^{d-2}}\, dr.
$$ 
Then 
$$
J(0) = \omega_d \displaystyle\int_0^R q^d \, r^{3-d} \vf^2 \, dr.
$$ 
Thus, it suffices to show that $|J(\ep)-J(0)| = \mathcal{O}(\ep^\frac{4-d}{2})$.  We can write
$$
|J(\ep)-J(0)|  = \omega_d \displaystyle\int_0^R q^d\left[ (\vf^2-1)+1\right] r^{d-1}\displaystyle\int_0^\ep \frac{d-2}{(a+r^2)^{d-1}}\, da\, dr.
$$
Let 
\begin{equation}\label{eq:j1}
J_1(\ep) = \displaystyle\int_0^\ep \displaystyle\int_0^R \frac{q^d \, r^{d-1}}{(a+r^2)^{d-1}}\, dr\, da,
\end{equation}
and
$$ 
J_2(\ep) = \displaystyle\int_0^R (\vf^2-1)q^d \, r^{d-1}\displaystyle\int_0^\ep \frac{1}{(a+r^2)^{d-1}}\, da\, dr.
$$
Making the change of variables $ r  =s \, \sqrt{a}$ in the inner integral of equation (\ref{eq:j1}) we have that 
$$ 
J_1(\ep) \le 2^d \displaystyle\int_0^\ep a^\frac{2-d}{2} \displaystyle\int_0^{\infty} \frac{s^{d-1}}{(1+s^2)^{d-1}}\,ds\, da.
$$
Since $2<d<4$ it follows that $J_1(\ep) = \mathcal{O}(\ep^\frac{4-d}{2})$.

\bigskip

\noindent
Moreover, since $\vf^2-1 \le C \, r^2$,  it follows that if $d<4$, then
$$ 
J_2(\ep) \le C \, \displaystyle\int_0^R q^d \, r^{d+1}\displaystyle\int_0^\ep \frac{1}{(a+r^2)^{d-1}}\, da\, dr \le 
C \,  2^d\, \ep \displaystyle \int_0^R r^{3-d}\, dr = \mathcal{O}(\ep).
$$
Thus, and since $2<d<4,$ it follows that $|J(\ep)-J(0)| = \mathcal{O}(\ep^\frac{4-d}{2})$.
\end{proof}

\bigskip

\begin{claim}\label{cl:tercera}
$$
\left(\omega_d \,  \displaystyle\int_0^R r^{d-1} q^d\ue^\frac{2d}{d-2}\, dr  \right)^\frac{d-2}{d}= 
\omega_d^\frac{d-2}{d}2^{d-2}\epsilon^\frac{2-d}{2}D_d^\frac{d-2}{d}+ \mathcal{O}(\ep^\frac{4-d}{2}),
$$
where 
$$
D_d = \frac{1}{2}\frac{\Gamma\left(\frac{d}{2}\right)^2}{\Gamma(d)}.
$$ 
\end{claim}

\begin{proof}
Let 
$$
K(\ep) =  \omega_d \, \displaystyle\int_0^R r^{d-1} q^d \ue^\frac{2d}{d-2}\, dr  = \omega_d \,  \displaystyle\int_0^R r^{d-1} q^d \frac{\vf^\frac{2d}{d-2}}{(\ep+r^2)^d}\, dr.
$$
Then, and since $ q^d \vf^\frac{2d}{d-2} = q^d(\vf^\frac{2d}{d-2}-1) + (q^d-2^d) + 2^d$,
we can write $K(\ep)=\omega_d \, ( K_1(\ep)+K_2(\ep)+K_3(\ep))$, where
$$
K_1(\ep) = \displaystyle\int_0^R \frac{q^d \, r^{d-1}}{(\ep+r^2)^d}(\vf^\frac{2d}{d-2}-1) \, dr;
$$
$$ 
K_2(\ep) =\displaystyle\int_0^R \frac{r^{d-1}(q^d-2^d)}{(\ep+r^2)^d} \, dr;
$$
and 
$$ 
K_3(\ep) = 2^d \, \displaystyle\int_0^\ep \frac{r^{d-1}}{(\ep+r^2)^d}\, dr.
$$

\bigskip

\noindent
Since $\vf(0)=1$ and $\vf'(0)=1$ it follows that $\vf^\frac{2d}{d-2}-1 \le C \, r^2.$ Thus, making the change of variables $r=s \, \sqrt{\ep},$ and since $d>2,$ it follows that
\begin{equation}\label{eq:K11}
K_1(\ep) \le C \,  2^d \, \displaystyle\int_0^R \frac{r^{d+1}}{(\ep+r^2)^d}\, dr \le C \,  2^d \, \ep^\frac{2-d}{2} \, 
\displaystyle\int_0^\infty \frac{s^{d+1}}{(1+s^2)^d}\, ds = \mathcal{O}(\ep^\frac{2-d}{2}).
\end{equation}
In order to obtain an estimate for $K_2(\ep)$, notice that if $0\le r \le R$,  then $0\le 2^d - q^d \le 2^d \,  A(R) \, r^2,$ where $A(R) = d(1+R^2)^{d-1}$.
Thus, 
$$
|K_2(\ep)| \le 2^d \,  A(R) \,  \displaystyle\int_0^R \frac{r^{d+1}}{(\ep+r^2)^d}\, dr.
$$ 
As before, we can make the change of variables $r = s \, \sqrt{\ep}$ to obtain 
\begin{equation}\label{eq:K2}
|K_2(\ep)| \le 2^d \, A(R) \ep^\frac{2-d}{2}\displaystyle\int_0^R \frac{s^{d+1}}{(1+s^2)^d}\, ds = \mathcal{O}(\ep^\frac{2-d}{2}).
 \end{equation} 
Finally, we will show that 
\begin{equation}\label{eq:K3}
K_3(\ep) = 2^d \, \ep^{-d/2} \, D_d + \mathcal{O}(1).
\end{equation}
In fact, making the change of variables $r=s \, \sqrt{\ep}$ we have that 
$$ 
K_3(\ep) = 2^d \ep^{-d/2} \left( \displaystyle\int_0^{\infty}\frac{s^{d-1}}{(1+s^2)^d}\, ds - \displaystyle\int_{\frac{R}{\sqrt{\ep}}}^\infty \frac{s^{d-1}}{(1+s^2)^d}\, ds\right). 
$$
But by equations (\ref{eq:casigama}) and (\ref{eq:errorcola}) it follows that 
$$ 
\displaystyle\int_0^{\infty}\frac{s^{d-1}}{(1+s^2)^d}\, ds = D_d,
$$
and 
$$
\displaystyle\int_{\frac{R}{\sqrt{\ep}}}^\infty \frac{s^{d-1}}{(1+s^2)^d}\, ds = \mathcal{O}(\ep^\frac{d}{2}).
$$
It follows from equations (\ref{eq:K11}), (\ref{eq:K2}) and (\ref{eq:K3}) that 
$$
K(\ep) = 2^d \,  \omega_d \, \ep^{-d/2} \, D_d + \mathcal{O}(\ep^{d/2}),$$ 
and so
$$ 
K(\ep)^\frac{d-2}{d} = \omega_d^\frac{d-2}{d}2^{d-2}\epsilon^\frac{2-d}{2} \, D_d^\frac{d-2}{d}+ \mathcal{O}(\ep^\frac{4-d}{2}).
$$
\end{proof}

\bigskip

\noindent
Recall that our goal is to show that if $\lambda >\frac{1}{4} [(2 \ell_2 +1)^2 - (d-1)^2]$,  then 
$Q_{\lambda}(\ue) < S$, where $S$ is the critical  Sobolev constant and $Q_{\lambda}$ is given by (\ref{eq:Qr}).

\bigskip 

\noindent
From the estimates obtained in Claim \ref{cl:primera}, Claim \ref{cl:segunda} and Claim \ref{cl:tercera} it follows that 
\begin{equation}
\begin{split}
& Q_{\lambda}(\ue)  = d(d-2)(\omega_d \, D_d)^\frac{2}{d} + \ep^\frac{d-2}{2} \, C_d \left[ \displaystyle\int_0^R r^{3-d}
\left(q^{d-2}\vf'^2 - (d-2)^2q^{d-1}\vf^2-\lambda q^d\vf^2 \right)\, dr \right] +\mathcal{O}(\ep),\\ 
\end{split}
\end{equation}
where $C_d = \omega_d^{2/d} \, 2^{2-d} \, D_d^\frac{2-d}{d}$. 
Notice that 
$$d(d-2)(\omega_d \, D_d)^\frac{2}{d}=\pi \,  d (d-2) \left( \frac{\Gamma\left(\frac{d}{2}\right)}{\Gamma(d)}\right)^\frac{2}{d},
$$
which is precisely the Sobolev critical constant $S$ (see, e.g., \cite{Ta76}, with $p=2,$ $m=d$ and $q= {2d}/{(d-2)}$). Now let 
$$
T(\vf) = \displaystyle\int_0^R r^{3-d}\left(q^{d-2}\vf'^2 - (d-2)^2 \, q^{d-1}\, \vf^2-\lambda q^d \, \vf^2 \right)\, dr.
$$
It suffices to show that $T(\vf)$ is negative if $\lambda >\frac{1}{4} [(2 \ell_2 +1)^2 - (d-1)^2]$.
 In order to conclude the proof we choose $\varphi=\varphi_1$, where $\varphi_1$ is the minimizer of 
$$
M(\varphi) = \displaystyle\int_0^R r^{3-d}\left(q^{d-2}\vf'^2 - (d-2)q^{d-1}\vf^2 \right)\, dr
$$
subject to the constraint
$$ 
\displaystyle\int_0^R r^{3-d} \,  q^d \, \vf^2 \, dr =1.
$$
The minimizer of $M(\varphi)$,  $\varphi_1,$ satisfies the Euler equation 
\begin{equation}\label{eq:conmu}
 -\frac{d}{dr}\left( r^{3-d} \, q^{d-2}\varphi_1'\right) - (d-2) \, r^{3-d} \, q^{d-1}\varphi_1 = \mu q^d \,  r^{3-d}\varphi_1.
\end{equation}
Multiplying (\ref{eq:conmu}) by $\varphi_1(r)$ and integrating between $0$ and $R$ we get, after integrating by parts, 
$$ 
\int_0^R r^{3-d} \, q^{d-2} \, \varphi_1'^2\, dr - (d-2) \int_0^R r^{3-d} \, q^{d-1} \, \varphi_1^2\, dr = \mu \int_0^R q^d \,  r^{3-d} \varphi_1^2\, dr.
$$ 
Thus, since 
$$\int_0^R q^d \,  r^{3-d}\varphi_1^2\, dr=1,
$$ 
$M(\varphi_1) = \mu$; hence,
$$ T(\varphi_1) = M(\varphi_1) - \lambda  = \mu-\lambda < 0
$$
if $\lambda>\mu$.  It suffices to show that $ \mu = \frac{1}{4}\left[ (2\ell_2+1)^2 - (d-1)^2\right],$ where $\ell_2$ is the first positive value for which the associated Legendre function $P_{\ell}^\frac{(d-2)}{2}(\cos\theta_1)$ vanishes.  Changing coordinates (setting $r= \tan\theta/2,$ so that $q=2\cos^2 \theta/2$) and letting
$$ 
\varphi_1(\t) = \sin^b\left(\frac{\t}{2}\right)\sin^a\left(\t\right)v(\theta),
$$
where $b= 2d-4$ and $a = \frac{1}{2}(6-3d)$ we obtain the equation for $v$ 
\begin{equation}\label{eq:legendreex}
 \ddot v(\t) + \cot\t \,\dot v(\t) + \left(\mu + \frac{d(d-2)}{4} - \frac{(d-2)^2}{4\sin^2\t}\right)v = 0, 
 \end{equation}
 with boundary condition $v(\t_1) = 0$. 

\begin{remark}
Equation (45)  is the same equation that determines the first Dirichlet eigenvalue of the original problem (i.e., equation (\ref{eq:a3})). 
We choose $a$ and $b$ precisely so that these two equations coincide. 
\end{remark}

The solutions of equation (\ref{eq:legendreex}) are $P_{\ell}^{\alpha}$ and $P_{\ell}^{-\alpha}$,  where $\alpha = \frac{2-d}{2}$ and 
$\ell(\ell+1) =  \mu + \frac{d(d-2)}{4}$. That is, $\ell = \frac{1}{2} \left(\sqrt{1+4 \mu - 4 \alpha + 4 \alpha^2} \, -1\right)$, and so
$$ \mu = \frac{1}{4}\left[(2\ell+1)^2 - (d-1)^2 \right].
$$
It follows that $\vf_1$ is of the form 
$$
{\vf_1} =  \sin^{b}\left(\frac{\t}{2}\right)\, \sin^{a}\t (A \, P_{\ell}^{\alpha} + B \,  P_{\ell}^{-\alpha}),
$$ 
where the choice of $A$ and $B$ must ensure the regularity of the solution. Notice that from the definition of $a$ and $b$ we have that $a+b=(d-2)/2$. 
Moreover, $\alpha = (2-d)/2$. Since $2<d$, we see that in order to have regular solutions at the origin we have to choose $A=0$. 
Finally, to satisfy the boundary condition $u(\theta_1)=0$  we must choose $\ell = \ell_2$. which finishes the proof of the lemma. Notice that, by Lemma \ref{lemma2.1},
$\ell_2< \ell_1$.
\end{proof}

\begin{remark} 
It is important to notice that what we actually present in the previous proof is the fact that  
for every $2<d<4$ and $\lambda > \frac{1}{4} [(2 \ell_2 +1)^2 - (d-1)^2]$ there is a minimizer of $Q_\lambda(u)$ (see equation (\ref{eq:Qr})). 
Given the invariance of $Q_{\lambda}(u)$ under the transformation $u \to \beta u$ (for any positive constant $\beta$), in order to get the Euler equation we minimize the numerator of (\ref{eq:Qr}) subject to the constraint $ \omega_d \,  \int_0^R  u^{p+1} \, q^d \, r^{d-1} \, dr=1$. The corresponding Euler equation is given by, 
\begin{equation}
- r^{d-1} (q^{d-2} u')'  - \lambda r^{d-1} \, q^d \, u = \mu  \,  r^{d-1} \, q^d \, u^p,
\label{eq:rmk1}
\end{equation}
where $\mu$ is a Lagrange multiplier. Multiplying (\ref{eq:rmk1}) by $\omega_d \, r^{d-1} \, u$, integrating in $r$ from $0$ to $R$, using the constraint and the characterization of 
$\lambda_1$ (i.e., the first Dirichlet eigenvalue), we have that, 
\begin{equation}
\mu \ge (\lambda_1-\lambda) \,  \omega_d \, \int_0^R  q^d \, u^2 \, r^{d-1} \, dr > 0, 
\label{eq:rmk2}
\end{equation}
provided $\lambda < \lambda_1$. In this case, if we set $u = \mu^{-1/(p-1)} \, v$, we finally see that $v$ solves 
\begin{equation}
- r^{d-1} (q^{d-2} u')'  - \lambda r^{d-1} \, q^d \, u = r^{d-1} \, q^d \, u^p.
\label{eq:rmk3}
\end{equation}
Going back to geodesic coordinates, i.e., $r \to \theta$, with $r= \tan \theta/2$, (\ref{eq:rmk3}) becomes (\ref{eq:ODE1}). 
From here it follows that  for any $2<d<4$, provided 
$$
\frac{1}{4} [(2 \ell_2 +1)^2 - (d-1)^2] < \lambda < \frac{1}{4} [(2 \ell_1 +1)^2 - (d-1)^2], 
$$
we have  the existence of a unique positive solution of (\ref{eq:ODE1}). 
\end{remark}

\bigskip
\bigskip

\section{Nonexistence of solutions} \label{sec:noexistencia}

\noindent
In this section we use a Rellich--Pohozaev  \cite{Po65,Re40} type argument to prove the nonexistence of regular positive solutions of the Boundary Value Problem
\begin{equation}
-u''-(d-1) \cot \theta \, u' = u^p + \lambda \, u
\label{eq:4.1}
\end{equation}
in the interval $(0,\theta_1)$, with boundary conditions $u'(0)=0$, $u(\theta_1)=0$ for a sharp range of values of $\lambda$. Here $2<d<4$ and $p=(d+2)/(d-2)$ is 
the critical Sobolev exponent. Here we will distinguish two separate sets of values of the parameters $\lambda$ and $\theta_1$:

\bigskip
\noindent
{\bf First Case:}  $\lambda > -d(d-2)/4$ and $0 \le \theta_1 \le \pi$. 

\bigskip
\noindent
{\bf Second Case:} $\lambda \le - d(d-2)/4$ and $0 \le \theta_1 \le \pi/2$. 

\noindent
\begin{remark} There is still a further case, namely $\lambda \le - d(d-2)/4$ and $\pi/2 < \theta_1 \le \pi$. In dimension $d=3$  the study of this case was initiated 
by Bandle and Benguria  \cite{BaBe02}, who showed numerically that there is a curve (denoted by  $\nu(\theta_1)$ in  \cite{BaBe02}) such that for $\lambda > \nu(\theta_1)$ 
there are no positive radial solutions of (\ref{eq:4.1}). On the other hand, in dimension $d=3$ for values of $\lambda$ below $\nu (\theta_1)$ there is a rich family of solutions.
These solutions were extensively studied in \cite{BaWe07, BaWe08, BrPe04, BrPe06}. For the whole interval $2<d<4$, we will explore the existence, nonexistence and multiplicity 
of positive solutions of (\ref{eq:4.1}) in a further publication. 
\end{remark}

\bigskip
\noindent
Our main results in this section are  the Lemmas \ref{Case1} and \ref{Case2} below. 

\begin{lemma} \label{Case1}
Assume $\lambda > -d(d-2)/4$ and $0 \le \theta_1 \le \pi$. Let $\ell_2$ be the first positive value of $\ell$ for which the associated Legendre function ${\rm P}_{\ell}^{(d-2)/2} (\cos\theta_1)$ vanishes. Then if
$$
\lambda \le \frac{1}{4} [(2 \ell_2 +1)^2 - (d-1)^2],
$$
there are no positive solutions of
\begin{equation}
-\frac{(\sin^{d-1}\t \,u')'}{\sin^{d-1}\t} = u^p +\lambda u,
\label{eq:4.2}
\end{equation} 
with boundary conditions $u'(0)=0$, and $u(\t_1) = 0$.
\end{lemma}
\begin{remark}
Notice that we have recast equation (\ref{eq:4.1}) in the form (\ref{eq:4.2}) which is more suitable in our proof. 
\end{remark}
\begin{proof}
Multiplying equation (\ref{eq:4.2}) by $g(\t)u'(\t) \sin^{2d-2}\t$, where $g(\theta)$ is a sufficiently smooth, nonnegative function defined in the interval $(0,\theta_1)$ 
satisfying the boundary conditions $g(0)=g'(0)=0$,  we obtain 
$$
 -\displaystyle\int_0^{\t_1} (\sin^{d-1}\t  \, u')'\,  u' \, g \, \sin^{d-1}\t \, d\t  =
  \displaystyle\int_0^{\t_1} \left( \frac{u^{p+1}}{p+1}\right)' \, g \, \sin^{2d-2} \t \, d\t + 
  \lambda \, \displaystyle\int_0^{\t_1} \left( \frac{u^2}{2}\right)'  \, g \, \sin^{2d-2}\t\,d\t.
$$
Integrating all the terms by parts, using the boundary conditions, we have that 
\begin{equation}
\begin{split} \label{eq:eng}
&\displaystyle\int_0^{\t_1} u'^2 \left( \frac{g'}{2}\sin^{2d-2}\t\right) \, d\t + \displaystyle\int_0^{\t_1} \frac{u^{p+1}}{p+1} \left( g'\sin^{2d-2}\t + g(2d-2) \sin^{2d-3}\t\cos\t\right) \,d\t\\
 &\,+ \lambda \displaystyle\int_0^{\t_1} \frac{u^2}{2} \left( g'\sin^{2d-2}\t + g(2d-2) \sin^{2d-3}\t\cos\t\right)\,d\t = \frac{1}{2} \sin^{2d-2}\t_1 u'(\t_1)^2g(\t_1).\\
\end{split}
\end{equation}

On the other hand, setting $h = \frac{1}{2} g' \, \sin^{d-1}\t $ and multiplying equation (\ref{eq:4.2}) by $h(\t) \, u(\t)\sin^{d-1}(\t)$ we obtain 
$$ 
-\displaystyle\int_0^{\t_1} (\sin^{d-1}\t u')' hu\,d\t = \displaystyle\int_0^{t_1} hu^{p+1}\sin^{d-1}\t\,d\t +\lambda \displaystyle\int_0^{\t_1}hu^2 \sin^{d-1}\t\,d\t.
$$ 
Integrating by parts we obtain 
\begin{equation}
\begin{split}\label{eq:enh}
& \displaystyle\int_0^{\t_1} u'^2h\sin^{d-1}\t \, d\t = \displaystyle\int_0^{\t_1}u^{p+1}h\sin^{d-1}\t\,d\t \\
&\, + \displaystyle\int_0^{\t_1} u^2 \left( \lambda h \sin^{d-1}\t +\frac{1}{2}h''\sin^{d-1}\t + \frac{1}{2}h'(d-1)\sin^{d-2}\t\cos\t\right)\,d\t. \\
\end{split}
\end{equation}
Notice that by our choice of $h$, the coefficient of $u'^2$ in equation (\ref{eq:eng}) is the same as the coefficient of $u'^2$ in equation (\ref{eq:enh}). Finally, subtracting equation (\ref{eq:eng}) from equation (\ref{eq:enh}) we obtain 

\begin{equation}\label{eq:poho}
\frac{1}{2}\sin^{2d-2}\t_1 u'(\t_1)^2g(\t_1) = \displaystyle\int_0^{\t_1} B\,u^{p+1}\, d\t + \displaystyle\int_0^{\t_1} A\, u^2\, d\t, 
\end{equation}
where 
\begin{equation}\label{eq:A}\begin{split}
A \equiv &\,\lambda \left(h\sin^{d-1}\t + \frac{1}{2} g' \sin^{2d-2}\t + g(d-1) \sin^{2d-3}\t\cos\t\right)\\
&\,  + \frac{1}{2}h''\sin^{d-1}\t + \frac{1}{2}h'(d-1) \sin^{d-2}\t\cos\t, \\
\end{split}
\end{equation}
and 
\begin{equation}\label{eq:B}
\begin{split}
B \equiv h \, \sin^{d-1}\t +\frac{g' \, \sin^{2d-2}\t }{p+1} + \frac{(2d-2) \, g\sin^{2d-3}\t\cos\t}{p+1} .
\end{split}
\end{equation}
Since by hypothesis $g(\t_1)\ge 0,$ it follows that the left hand side of equation (\ref{eq:poho}) is nonnegative. In the sequel (see the Claim \ref{cl:mm} and the Lemma 
\ref{LM} below), 
we show that for any
$$
\lambda \le \frac{1}{4} [(2 \ell_2 +1)^2 - (d-1)^2],
$$
there exists a choice of $g$ so that
$A \equiv 0$,  and $B$ is negative. That is, we will show that for that range of $\lambda$'s the right hand side of equation (\ref{eq:poho}) is negative, thus obtaining a contradiction. 
\end{proof}

\bigskip 

\noindent
Substituting $h  = \frac{1}{2} g' \, \sin^{d-1}\t $ in equation (\ref{eq:A}) we obtain
\begin{equation}\label{eq:A2}\begin{split}
A =&\, \sin^{2d-2}\t \left[ \frac{g'''}{4} + \frac{3}{4}g''(d-1)\cot\t  \right.\\
&\,+ \left. g' \left( \frac{(d-1)(2d-3)\cot^2\t}{4} - \frac{d-1}{4}+\lambda\right) +\lambda g(d-1)\cot\t\right].  \\
\end{split}
\end{equation}
Finally, making the change of variables $g= f/\sin^2\t$ we obtain
\begin{equation}\label{eq:A3}\begin{split}
A = \sin^{2d-4}\t &\, \left[ \frac{f'''}{4} +\frac{3}{4}(d-3)\cot\t f'' + f' \left( \frac{(d-3)(2d-11)}{4}\cot^2\t +\frac{7-d}{4} +\lambda \right) \right.\\
 &\,\,\,\,  \left. + f \left( (d-3)(4-d) \cot^3\t + 2(d-3)\cot \t +\lambda (d-3)\cot\t\right) \right]. \\
\end{split}
\end{equation}

\begin{claim}\label{cl:mm}
For any $2<d<4$, the function  
 $$ 
 z(\t) = \sin^{4-d}\t P_\ell^\alpha(\cos\t) P_\ell^{-\alpha}(\cos\t),
 $$ 
 with $\alpha = (2-d)/2$ and $ \ell = \frac{1}{2}\left( \sqrt{4\lambda + (d-1)^2}-1\right)$,
is a solution of
\begin{equation}\label{eq:noex1}
\begin{split}
& \frac{f'''}{4} +\frac{3}{4}(d-3)\cot\t f'' + f' \left( \frac{(d-3)(2d-11)}{4}\cot^2\t +\frac{7-d}{4} 
+ \lambda \right)\\ +&\, f \left( (d-3)(4-d) \cot^3\t + 2(d-3)\cot \t +\lambda (d-3)\cot\t\right) = 0. \\ 
\end{split}
\end{equation}
 \end{claim}

\begin{proof}
Let $y_1(\t) = P_\ell^{\alpha}(\cos\t)$ and $y_2(\t) = P_\ell^{-\alpha}(\cos\t).$ Then $y_1$ and $y_2$ are solutions to 
\begin{equation}\label{eq:legendre}
y''(\t) + \cot\t  \, y' (\t)+ k(\t)y(\t) = 0,
\end{equation}
where 
\begin{equation}
k(\t) = \ell(\ell+1)-\frac{\alpha^2}{\sin^2\t}.
\label{eq:4.3}
\end{equation} 
Let $v(\t) = y_1(\t) \, y_2(\t)$. Then, it follows from (\ref{eq:legendre}) that
$$
y_1''y_2 + y_2'' y_1 = -\cot\t v' -2kv,
$$ 
which in turn implies
$$ 
v'' = -2kv -\cot \t \,  v' + 2y_1'y_2'.
$$
Similarly, and since 
$$ 
y_1''y_2' +y_1'y_2'' = -2\cot\t \, y_1'y_2' -kv',
$$
we obtain 
\begin{equation}\label{eq:nex3}
 v''' +3\cot \t \,  v'' + v' \left( 4k - \csc^2\t+2\cot^2\t\right) +4v\left( \alpha^2\cot\t\csc^2\t+k\cot\t\right) =0.
 \end{equation}
Now, we make the change of variables $v \to f$ given by
$$ 
f(\t) = \sin^{4-d}\t \,v(\t)
$$
in equation (\ref{eq:nex3}) and multiply the resulting equation through by $\sin^{d-4}\t$. 
Setting $\alpha=(2-d)/2$,  $ \ell = \frac{1}{2}\left( \sqrt{4\lambda + (d-1)^2}-1\right)$ (which is the positive solution of $4\ell(\ell+1)=4 \lambda + d^2-2d$) and, using
(\ref{eq:4.3}), we see that $f$ satisfies (\ref{eq:noex1}). 
This finishes the proof of Claim \ref{cl:mm}.
\end{proof}

\begin{remark} Notice that in order to ensure that $\ell = \left( \sqrt{4\lambda + (d-1)^2}-1\right)/2$ is positive, we need to have that $4 \lambda + d(d-2) >0$, 
which we have assumed in the hypothesis of Lemma \ref{Case1}. 
\end{remark}

\begin{lemma}\label{LM}
Let $\alpha = (2-d)/2$,  $ \ell = \frac{1}{2}\left( \sqrt{4\lambda + (d-1)^2}-1\right),$ and $\ell_2$ be the first positive value of $\ell$ for which 
 $P_\ell^{\alpha}(cos \theta_1)$ vanishes.  Consider
 \begin{equation}
B \equiv h\sin^{d-1}\t +\frac{g' \sin^{2d-2}\t }{p+1} + \frac{(2d-2)g\sin^{2d-3}\t\cos\t}{p+1} ,
\label{eq:4.4}
\end{equation}
where $h(\t) = \frac{1}{2} g' \sin^{d-1}\t$,  $g(\t)= f(\t)\sin^{-2}\t$ and $f(\t)= \sin^{4-d}\t P_\ell^\alpha(\cos\t)P_\ell^{-\alpha}(\cos\t)$. Then $B$ is negative on $[0,\ell_2)$. 
\end{lemma}

\begin{proof}
The associated Legendre functions satisfy the following raising and lowering relations (see, e.g., \cite{AbSt72}, equation 8.1.2, pp. 332), 
which we will use repeatedly in the proof of this lemma:
\begin{equation}\label{eq:subida}
\dot P_\ell^{\alpha}(\cos\t) = \frac{-P_\ell^{\alpha+1}}{\sin\t} - \frac{\alpha\cos\t P_\ell^\alpha}{\sin^2\t},
\end{equation}
and 
\begin{equation}\label{eq:bajada}
\dot P_\ell^{\alpha+1}(\cos\t)   =\frac{1}{\sin^2\t}\left( (\ell+\alpha+1)(\ell-\alpha)\sin\t P_\ell^\alpha +(\alpha+1)\cos\t P_\ell^{\alpha+1}\right).
\end{equation}
Notice that in the two previous equations, $\dot P_\ell^{\alpha}$ means the derivative of $P_\ell^{\alpha}$ with respect to its argument, therefore, 
$$
\frac{d}{d\theta} P_\ell^{\alpha}(\cos \theta) = - \sin \theta \, \dot P_\ell^{\alpha}(\cos \theta).
$$
After substituting for $h, g$ and $f$ we can write
\begin{equation}\label{eq:bnuevo}
B =  - \left( \frac{d-1}{d} \right) \sin^{d+1}\t  \, (\dot P_\ell^{\alpha} P_\ell^{-\alpha}+ P_\ell^\alpha\dot P_\ell^{-\alpha}). 
\end{equation} 
Hence, it suffices to show that $\dot P_\ell^{\alpha} P_\ell^{-\alpha}+ P_\ell^\alpha\dot P_\ell^{-\alpha} > 0$ on $[0,\ell_2)$.
Because of Lemma 2.1, $P_\ell^{\alpha}P_\ell^{-\alpha}$ is positive, on this interval. 
Thus, we can write this inequality as
\begin{equation}
\frac{\dot P_\ell^{\alpha}}{P_\ell^\alpha} + \frac{\dot P_\ell^{-\alpha}}{ P_\ell^{-\alpha}} > 0.
\label{eq:4.5}
\end{equation}
It follows from equation (\ref{eq:subida}) that 
\begin{equation}
\frac{\dot P_\ell^{\alpha}}{P_\ell^\alpha} + \frac{\dot P_\ell^{-\alpha}}{ P_\ell^{-\alpha}} =-\frac{1}{\sin\t}\frac{P_\ell^{\alpha+1}}{P_\ell^\alpha} - \frac{1}{\sin\t}\frac{P_\ell^{-\alpha+1}}{P_\ell^{-\alpha}}.
\label{eq:4.6}
\end{equation}
Given, the identity (\ref{eq:4.6}) above, in order to prove (\ref{eq:4.5}) it is convenient to introduce the function,
\begin{equation}\label{eq:defdey}
y_{\nu}(\t) = \frac{-1}{\sin\t}\frac{P_\ell^{\nu+1}(\cos\t)}{P_\ell^{\nu}(\cos\t)} - \frac{\nu}{2\sin^2\frac{\t}{2}}.
\end{equation}
In the sequel, we  study the behavior of $y_{\nu}(\t)$ on $[0,\ell_2)$.
 In particular, we will show that $y_{\nu}$ is positive on this interval if $-1<\nu<1.$ This in turn will imply that  
$$
\frac{\dot P_\ell^{\alpha}}{P_\ell^\alpha} + \frac{\dot P_\ell^{-\alpha}}{ P_\ell^{-\alpha}} = y_{\alpha}(\t) + y_{-\alpha}(\t)>0.
$$

\bigskip 

\bigskip 

\noindent
Using the series expansion of the Associated Legendre  functions given by (\ref{Taylor}) in terms of the Hypergeometric Function (\ref{hyper}) we readily get, 
$$ 
P_\ell^\nu (\cos\t) = \frac{1}{\Gamma(1-\nu)}\cot^\nu\frac{\t}{2}\left( 1 - \frac{\ell(\ell+1)}{1-\nu}\sin^2\frac{\t}{2}+ \frac{\ell(\ell^2-1)(\ell+2)}{2(1-\nu)(2-\nu)}\sin^4\frac{\t}{2} +\mathcal{O}\left( \sin^6\frac{\t}{2}\right) \right) .
$$ 
It follows that 
\begin{equation}\label{eq:cuocientep}
 \frac{P_\ell^{\nu+1}(\cos\t)}{P_\ell^\nu(\cos\t)}= \frac{\Gamma(1-\nu)}{\Gamma(-\nu)}\cot\frac{\t}{2}\left(  1 + E\sin^2\frac{\t}{2} +\mathcal{O}\left( \sin^4\frac{\t}{2}\right)\right),
\end{equation}
where 
$$ 
E  = \frac{\ell(\ell+1)}{\nu(1-\nu)}
$$
(here we used that $\Gamma(1-\nu) = - \nu \Gamma(-\nu)$). 
Thus, it follows from equations (\ref{eq:defdey}) and (\ref{eq:cuocientep}) that 
$$ 
y_{\nu}(\t) = \frac{\nu}{2}\left( E+ \mathcal{O}\left( \sin^2\frac{\t}{2}\right) \right). 
$$
In particular, 
$$
 \lim_{\t  \to  0}y_{\nu}(\t)  = \frac{\ell(\ell+1)}{2(1-\nu)} >0,
 $$ 
since we are considering $\ell > 0$ and $-1<\nu <1.$ We will now show by contradiction that there is no point on the interval $[0,\ell_2)$ where $y_{\nu}$ changes sign. To do so, we first derive a Riccati equation for $y_{\nu}$.  
It follows from equation (\ref{eq:defdey}) that 
\begin{equation}\label{eq:derivdey}
\dot y_{\nu}  = \frac{\cos\t}{\sin^2\t}\frac{P_\ell^{\nu+1}}{P_\ell^\nu} + \frac{\dot P_\ell^{\nu+1}}{P_\ell^\nu} - \frac{P_\ell^{\nu+1}\dot P_\ell^\nu}{\left(P_\ell^\nu\right)^2} +\frac{\nu(1+\cos\t)^2}{\sin^3\t}.
\end{equation}
Using equations (\ref{eq:subida}) and (\ref{eq:bajada}) in equation (\ref{eq:derivdey}) we obtain 
\begin{equation}\label{eq:ri1}
\dot y_{\nu} = \frac{1}{\sin\t}\left(\frac{P_\ell^{\nu+1}}{P_\ell^\nu}\right)^2 +\frac{2(\nu+1)\cos\t}{\sin^2\t}\frac{P_\ell^{\nu+1}}{P_\ell^\nu} + \frac{(\ell+\nu+1)(\ell-\nu)}{\sin\t}  +\frac{\nu(1+\cos\t)^2}{\sin^3\t}.
\end{equation}
Finally, using equation (\ref{eq:defdey}) to solve for ${P_\ell^{\nu+1}}/{P_\ell^\nu}$ we obtain the following Riccati equation for $y_{\nu}$,
\begin{equation}\label{eq:riccati}
\dot y_{\nu} = \sin\t \, y_{\nu}^2 +\frac{2y_{\nu}}{\sin\t}(\nu-\cos\t) + \frac{\ell(\ell+1)}{\sin\t}.
\end{equation}
Since $y_{\nu(0)} >0$, and $y_{\nu}(\theta)$ is continuous in $\theta$, 
if $y_{\nu}(\t)$ were to cross $y_{\nu}=0,$ there would exist a point, $\t^*,$ such that $y_{\nu}(\t^*)=0$ and $\dot y_{\nu}(\t^*)<0.$ But from equation (\ref{eq:riccati}) we would 
then have 
$$ 
\dot y_{\nu}(\t^*) = \frac{\ell(\ell+1)}{\sin\t^*} >0,
$$
arriving at a contradiction. We conclude that $y_{\nu}$ is positive on $[0,\ell_2)$. 
\end{proof}

\begin{lemma}[Case 2] \label{Case2}
Assume $\lambda \le  -d(d-2)/4$ and $0 \le \theta_1 \le \pi/2$. Then, there are no positive solutions of
\begin{equation}
-\frac{(\sin^{d-1}\t \,u')'}{\sin^{d-1}\t} = u^p +\lambda u,
\label{eq:4.50}
\end{equation} 
\noindent with boundary conditions $u'(0)=0$, and $u(\t_1) = 0$.
\end{lemma}

\begin{proof}
Proceeding as in the proof of (\ref{Case1}), we conclude that if $u$ is a smooth  positive solution of (\ref{eq:4.1}) satisfying the boundary conditions $u'(0)=0$ 
and $u(\theta_1)=0$, and $f$ is a  smooth function satisfying $f(0)=0$  we have
\begin{equation}\label{eq:pohonueva}
\frac{1}{2}\sin^{2d-4}\t_1 \, u'(\t_1)^2 \, f(\t_1) = \displaystyle\int_0^{\t_1} B\,u^{p+1}\, d\t + \displaystyle\int_0^{\t_1} A\, u^2\, d\t, 
\end{equation}
where 
\begin{equation}\label{eq:A3nueva}\begin{split}
A = \sin^{2d-4}\t &\, \left[ \frac{f'''}{4} +\frac{3}{4}(d-3)\cot\t f'' + f' \left( \frac{(d-3)(2d-11)}{4}\cot^2\t +\frac{7-d}{4} +\lambda \right) \right.\\
 &\,\,\,\,  \left. + f \left( (d-3)(4-d) \cot^3\t + 2(d-3)\cot \t +\lambda (d-3)\cot\t\right) \right]. \\
\end{split}
\end{equation}
and 
\begin{equation}\label{eq:Bnueva}
\begin{split}
B \equiv \frac{d-1}{d} \sin^{2d-4} \theta \left[f'(\theta) + (d-4) \cot \theta \, f(\theta) \right].
\end{split}
\end{equation}
Equations (\ref{eq:pohonueva}),  (\ref{eq:A3nueva}), and  (\ref{eq:Bnueva}), follow from (\ref{eq:poho}),  (\ref{eq:A}), and  (\ref{eq:B}) respectively setting $g(\theta)= 
f(\theta) /\sin^2 (\theta)$ (see in fact (\ref{eq:A2})).

\bigskip
\noindent
This time we choose $f(\theta) = \sin^{4-d} (\theta)$, so that $B \equiv 0$. After a long but straightforward computation we find that for this choice of $f(\theta)$ one has, 
\begin{equation}
A(f) = \frac{1}{4} \cot \theta \left[ 4 \lambda + d (d-2) \right] \le 0, 
\label{eq:borr6}
\end{equation}
provided   $4 \lambda + n (n-2) \le 0$ and $ \theta \le \pi/2$ (since $\cot \theta \ge 0$ if $\theta \le \pi/2$).  Using (\ref{eq:borr6}) in (\ref{eq:pohonueva}) we get a contradiction, which concludes the proof of this lemma. 
\end{proof}

\section*{Acknowledgements}

We would like to thank M.~Cristina Depassier for helping us with Mathematica to compute  the numerical curves contained in the figures we give in the Introduction.
We also thank the anonymous referee for the suggestions that helped us improve the presentation of the results of this manuscript. 
The work of RB has been supported by Fondecyt (Chile) Projects \# 112--0836, \# 116--0856, and \# 114--1155, and by the Nucleo Milenio en  ``F\ii sica Matem\'atica'', RC--12--0002 (ICM, Chile).

\bigskip
\section*{Appendix A: Proof of Theorems 1.1 and 1.4}

\begin{proof} [Proof of Theorem 1.1]

\noindent
In 1999 Jannelli \cite{Ja99} considered the non linear boundary value problem, 
\begin{equation}
-u'' - \frac{d-1}{r} \, u' = |u|^{4/(d-2)} \, u + \lambda \, u, 
\label{jannelli1}
\end{equation}
in the interval $(0,1)$ with the boundary conditions $u'(0)=0$ and $u(1)=0$, for $2<d<4$, where  $u>0$ is such that 
$\int_0^1 u^2(r) \, r^{d-1} \, dr < \infty$ and $\int_0^1 {u'}^2(r) \, r^{d-1} \, dr < \infty$. Among other results in \cite{Ja99} Jannelli proved  that
(\ref{jannelli1}) has a unique positive solution if 
\begin{equation}
j_{-(d-2)/2,1}^2 < \lambda < j_{(d-2)/2,1}^2, 
\label{jannelli2}
\end{equation}
and no positive solutions if $\lambda$ lies outside this interval (here $j_{k,\ell}$ denotes the $\ell$-th positive zero of the Bessel function $J_k(t)$). 
From the result of Jannelli it is simple to prove our Theorem 1.1. For radial solutions of (\ref{BNwithdrift1}), one is lead to the equation (\ref{jannelli1}), considered by Jannelli, provided 
we set $d=n - \delta$. Because of Hardy's inequality we have to constrain $\delta$ to  $\delta < (n-2)/2$. 
If $n=3$, $d=3 -\delta$ and $\delta <1/2$. Given that $d=3-\delta$, if $\delta \in (-1,1/2)$, we are in the situation 
given by (\ref{jannelli2}) and the proof of i) of Theorem 1.1 follows. The proofs of ii), and iii)
follow in a similar way.
\end{proof}

\bigskip
\bigskip
\noindent
The proof of  Theorem 1.4 follows from the results of our Theorem 1.2 following the same line of thought we used above to prove Theorem 1.1 from the results of Jannelli. 

\bigskip
\begin{proof} [Proof of Theorem 1.4]

\noindent
For radial solutions (i.e., solutions that only depend on the azimuthal angle $\theta$) of (\ref{BBDriftSn}) on a geodesic ball $D$ centered at the $N\!P$ of $\mathbb{S}^n$, $n\ge 3$, 
one is led to the boundary value problem (BVP)  given by (\ref{BBSnRadial}) on the interval $(0,\theta_1)$, with boundary conditions $u'(0)=0$ and $u(\theta_1)=0$. Here $\theta_1$ is
the (geodesic) radius of the geodesic ball $D$. In Theorem 1.2, we determined the exact ranges of $\lambda$ (depending on the values of the parameter $d >2$ and the 
values of $\theta_1$) for which there is a (unique) positive solution and the range of values of $\lambda$ for which there are no positive solutions of this BVP. 
In particular, if $2<d<4$, for $\lambda>-d(d-2)/4$ and all $0 < \theta_1 \le \pi$,  there are unique positive solutions in the range (\ref{rangeoflambda}) and no solutions 
outside that range. Again, because of Hardy's inequality we have the constraint $\delta < (n-2)/2$. 
If $n=3$, $d=3 -\delta$ and $\delta <1/2$. Given that $d=3-\delta$, if $\delta \in (-1,1/2)$, we are in the situation 
given by (\ref{rangeoflambda}) and the proof of i) of Theorem 1.4 follows. The proofs of ii), and iii) 
follow in a similar way.
\end{proof}

\bigskip

\section*{Appendix B: Uniqueness of positive solutions}

\noindent
Here we prove the uniqueness of positive solutions of (\ref{BBSnRadial}) where $2<d<4$ provided $\lambda \ge -d(d-2)/4$.  
The proof follows from the classical result of Kwong and Li \cite{KwLi92}. Consider the equation 
\begin{equation}
-u''-(d-1) \cot \theta \, u' = u^p + \lambda \, u
\label{eq:Ap1}
\end{equation}
in the interval $(0,\theta_1)$, where $u \ge 0$ satisfies the  boundary conditions $u'(0)=0$, $u(\theta_1)=0$, with $\theta_1 \le \pi$. 
Here $2<d<4$ and $p=(d+2)/(d-2)$ is  the critical Sobolev exponent. Let $\alpha=(2-d)/2$ as before, and make the change of variable $u \to v$ given by
\begin{equation}
u(\theta) = v(\theta) \sin^{\alpha}(\theta). 
\label{eq:Ap2}
\end{equation}
Then, equation (\ref{eq:Ap1}) becomes, 
\begin{equation}
 \sin^2 \theta \, v''(\theta)  +  \sin \theta \, \cos \theta \, v'(\theta)  +  {v(\theta)}^p + G_{\lambda}(\theta)  \, v (\theta)=0,
\label{eq:Ap3}
\end{equation}
where the function $G_{\lambda}(\theta)$ is given by 
\begin{equation}
 G_{\lambda}(\theta) = -\alpha^2 + \left[\lambda + \frac{d(d-2)}{4} \right] \sin^2 \theta. 
\label{eq:Ap4}
\end{equation}
Notice that the function $G_{\lambda}(\theta)$ is a {\it $\Lambda$ function}  in the sense of Kwong and Li \cite{KwLi92}, i.e., it first increases, it has at most one maximum and then decreases when $\theta$ runs from $0$ to $\theta_1$ and $\theta_1 \le \pi$ and $4 \lambda + d(d-2) \ge 0$. 

\bigskip
\noindent
Next define the {\it energy function} 
\begin{equation}
E [v] \equiv \sin^2 \theta \, {v'(\theta)}^2 + \frac{2}{p+1} v(\theta)^{p+1} + G_{\lambda}(\theta) \, v(\theta)^2.
\label{eq:Ap5}
\end{equation}
Then, if $v(\theta)$ solves (\ref{eq:Ap3}), we have that
\begin{equation}
\frac{dE}{d\theta} =G_{\lambda}' \, v^2. 
\label{eq:Ap6}
\end{equation}
Since for $\lambda \ge -d(d-2)/4$ and $\theta_1 \le \pi$, $G_{\lambda}(\theta)$ is a {\it $\Lambda$--function}, it follows from \cite{KwLi92} that $v$, hence $u$,  is unique. 
The condition $\lambda \ge -d(d-2)/4$, that ensures that $G_{\lambda}(\theta)$ is a {\it $\Lambda$--function}, is needed to prove uniqueness. In fact, if one drops this condition 
in the case $d=3$ uniqueness fails as shown in \cite{BaWe07,BaWe08,BrPe04,BrPe06}. Results similar to those obtained  in \cite{BaWe07,BaWe08,BrPe04,BrPe06} are expected for $2<d<4$ when $\lambda < -d(d-2)/4$.




\end{document}